\documentclass[10pt, twoside, reqno]{amsart}

\usepackage{geometry}

\usepackage[all]{xy}
\usepackage{stmaryrd}
\usepackage[usenames, dvipsnames]{color}
\usepackage[pdfborder={0 0 0}, colorlinks = true, linkcolor=blue, citecolor=OliveGreen]{hyperref}
\usepackage{graphicx}
\usepackage{amssymb}
\usepackage{amsthm}
\usepackage{amsmath, amscd}
\usepackage{mathrsfs}
\usepackage{paralist}
\usepackage{framed}
\usepackage{float}
\usepackage{fourier}

\usepackage{tikz, wrapfig}
\usetikzlibrary{trees}

\usepackage{sidecap}
\sidecaptionvpos{figure}{c}

\usepackage{amsfonts} 

\DeclareMathOperator{\Q}{\mathbb{Q}}
\DeclareMathOperator{\Z}{\mathbb{Z}}
\DeclareMathOperator{\C}{\mathbb{C}}
\DeclareMathOperator{\A}{\mathbb{A}}
\DeclareMathOperator{\R}{\mathbb{R}}
\DeclareMathOperator{\F}{\mathbb{F}}

\DeclareMathOperator{\D}{\mathcal{D}}
\DeclareMathOperator{\length}{length}

\DeclareMathOperator\End{End}
\DeclareMathOperator\Hom{Hom}
\DeclareMathOperator\Spec{Spec}
\DeclareMathOperator\Spf{Spf}

\newcommand{\Zed}{\mathfrak{Z}}
\newcommand{\X}{\mathbb{X}}
\newcommand{\Y}{\mathbb{Y}}
\newcommand{\V}{\mathbb{V}}

\newcommand{\M}{\mathcal M}
\DeclareMathOperator{\Diff}{Diff}
\DeclareMathOperator{\ord}{ord}
\DeclareMathOperator{\Herm}{Herm}
\DeclareMathOperator{\inv}{inv}
\DeclareMathOperator{\vol}{vol}
\DeclareMathOperator{\Ta}{\mathtt{Ta}}
\DeclareMathOperator{\Lie}{\mathsf{Lie}}
\DeclareMathOperator{\diag}{\mathsf{diag}}
\DeclareMathOperator{\Id}{\mathtt{Id}}

\newcommand{\la}{\langle} 	
\newcommand{\ra}{\rangle} 	
\newcommand{ \lie}[1]{\mathfrak{#1} } 
\newcommand{\isomto}{\overset{\sim}{\longrightarrow}}

\newcommand{ \sm}[1]{ \begin{smallmatrix} #1 \end{smallmatrix} }

\newcommand{ \refThm}[1]{\hyperref[#1]{Theorem \ref{#1}}}
\newcommand{ \refSec}[1]{\hyperref[#1]{Section \ref{#1}}}
\newcommand{ \refProp}[1]{\hyperref[#1]{Proposition \ref{#1}}}
\newcommand{ \refLemma}[1]{\hyperref[#1]{Lemma \ref{#1}}}
\newcommand{\refRemark}[1]{\hyperref[#1]{Remark \ref{#1}}}
\newcommand{\refDef}[1]{\hyperref[#1]{Definition \ref{#1}}}
\newcommand{\refCor}[1]{\hyperref[#1]{Corollary \ref{#1}}}

\theoremstyle{plain}
\newtheorem{theorem}{Theorem}[section]
\newtheorem*{theorem*}{Theorem}
\newtheorem*{maintheorem*}{Main Theorem}

\newtheorem{lemma}[theorem]{Lemma}
\newtheorem{proposition}[theorem]{Proposition}
\newtheorem{definition}[theorem]{Definition}
\newtheorem{corollary}[theorem]{Corollary}

\newtheorem*{claim*}{Claim}

\numberwithin{equation}{section}

\theoremstyle{remark}
\newtheorem{remark}[theorem]{Remark}

\newcommand{\Nilp}{\mathbf{Nilp}}

\usepackage{framed}
\renewenvironment{leftbar}[1][\hsize]
{%
    \MakeFramed{\hsize#1\advance\hsize-\width\FrameRestore}%
}
{\endMakeFramed}

\title{Improper Intersections of Kudla-Rapoport divisors and Eisenstein series}
\author{Siddarth Sankaran}
\begin{document}
\begin{abstract}
We consider a certain family of Kudla-Rapoport cycles on an integral model of a Shimura variety attached to a unitary group of signature (1,1), and prove that the arithmetic degrees of these cycles can be identified with the Fourier coefficients of the central derivative of an Eisenstein series of genus 2. The integral model in question parametrizes abelian surfaces equipped with a non-principal polarization and an action of an imaginary quadratic number ring, and in this setting the cycles are degenerate: they may contain components of positive dimension. This result  can  be viewed as confirmation, in the degenerate setting and for dimension 2, of conjectures of Kudla and Kudla-Rapoport that predict relations between the intersection numbers of special cycles and the Fourier coefficients of automorphic forms. 

\end{abstract}
\maketitle
\section{Introduction}
In their article \cite{KRunnglob}, Kudla and Rapoport investigate integral models of Shimura varieties attached to unitary groups of signature $(n-1,1)$. These models are defined as moduli spaces of abelian varieties equipped with an action of the maximal order $o_k$ in a fixed imaginary quadratic field, together with a compatible principal polarization. Kudla and Rapoport go on to construct a family of `special' cycles, and prove that when such a cycle is zero-dimensional and is supported in the fibre of an unramified prime, its  degree can be identified with a Fourier coefficient of the derivative of an incoherent Eisenstein series for $U(n,n)$ at its centre of symmetry. This result is in line with a deep conjectural programme, initiated by Kudla and supported by his collaborators, that aims to establish systematic relations between arithmetic cycles on Shimura varieties and Fourier coefficients of automorphic forms, cf. the survey article \cite{KudlaMSRI}.

In this paper, we study an extension of the problem of Kudla-Rapoport in the case $n=2$, where we allow the polarizations to be non-principal in a controlled way. Though the cycles in this setting might not be  zero-dimensional, our main result asserts that their  degrees, suitably defined, are again identified with the Fourier coefficients of the  central derivative of a (non-standard) Eisenstein series for $U(2,2)$. 

We now give a more precise account of this result. Let $k$ be an imaginary quadratic field, with ring of integers $o_k$, and fix a squarefree integer $d \in \Z_{>0}$, all of whose factors are inert primes in $k$. We define $\mathrm M^d_{(1,1)}$ to be the moduli stack of triples $\underline A = (A, i, \lambda)$, where $A$ is an abelian surface equipped with an action
\[ i \colon o_k \ \to \ \End(A) \]
that satisfies the \emph{signature $(1,1)$ condition}, cf.\ \refDef{Md(1,1)Def} below, and $\lambda$ is a polarization such that
\begin{compactenum}[(i)]
\item the corresponding Rosati involution induces Galois conjugation on the image $i(o_k)$; and 
\item $\ker(\lambda) \subset A[d]$ is contained in the $d$-torsion of $A$ with $|\ker(\lambda)| = d^2$. 
\end{compactenum}
This moduli problem is representable by a Deligne-Mumford stack that is flat over $\Spec(o_k)$. 
 
Next, we let $\mathcal E$ be the Deligne-Mumford stack parametrizing triples $\underline E = (E, i_E, \lambda_E)$ consisting of an elliptic curve $E$ with an $o_k$-action $i_E \colon o_k \to \End(E)$ satisfying the \emph{signature $(1,0)$} condition, and a principal polarization $\lambda_E$ whose Rosati again induces Galois conjugation on the image $i_E(o_k)$. 

Following \cite{KRunnglob}, we define the \emph{Kudla-Rapoport cycles} on the product $\mathcal M := \ \mathcal E \times_{o_k} \mathrm M^d_{(1,1)}$ as follows. Suppose we are given points $\underline E \in \mathcal E(S) $ and $\underline A \in  \mathrm M^d_{(1,1)}(S)$ valued in some connected base scheme $S$ over $\Spec(o_k)$. Then the space 
\[ \Hom_{S, o_k}( E, A) \]
of $o_k$-linear morphisms admits a positive definite $o_k$-hermitian form defined by the formula
\[ (x, y) \ : = \ \lambda_E^{-1} \ \circ \ y^{\vee} \  \circ \ \lambda_A \ \circ x \ \in \End_{o_k}(E) \ \simeq o_k .\]
Given an integer $m \in \Z$, we define  $\Zed(m)$ to be the moduli space of tuples
\[ \Zed(m)(S) \ = \ \{ (\underline E, \, \underline A, \, x ) \ | \ (\underline E, \underline A ) \in \mathcal M(S), \ x \in \Hom_{S, o_k}(E, A) \text{ with } (x, x) = m \}.\]
This moduli problem is representable by a DM stack, and the natural forgetful map $\Zed(m) \to \mathcal M$ is finite and unramified. We thereby obtain a cycle on $\mathcal M$, which, abusing notation, we denote by the same symbol $\Zed(m)$. 
 
Similarly, for any matrix $T \in \Herm_2(o_k)$, we define $\Zed(T)$ to be the moduli space of tuples
 \[ \Zed(T)(S) \ = \ \{ (\underline E, \, \underline A, \, \mathbf x ) \} \]
 where $(\underline E, \underline A) \in \mathcal M(S)$ as before, and $\mathbf x = [x_1, x_2] \in \Hom_{S, o_k}(E, A)^2$ is a pair of maps such that 
 \[ (\mathbf x , \, \mathbf x) \ := \ \begin{pmatrix} (x_1, x_1) & (x_1, x_2) \\ (x_2, x_1) & (x_2, x_2) \end{pmatrix} \  = \ T. \]
 As before, the forgetful map $\Zed(T) \to \mathcal M$ defines a cycle, denoted by the same symbol. 

Let $T = (\sm{ t_1 & * \\ * & t_2 } ) $. Then there is a decomposition
\[ \Zed(t_1) \ \times_{\mathcal M} \ \Zed(t_2) \ = \ \coprod_{T'  = (\sm{ t_1 & * \\ *&  t_2 }) } \ \Zed(T') \]
over all cycles $\Zed(T')$ corresponding to matrices $T'$ with the same diagonal entries as $T$. 

When $T$ is non-singular, it turns out that the generic fibre $\Zed(T)_k$ is empty, and so the support of $\Zed(T)$ is concentrated in finitely many fibres of non-zero characteristic. In this case, we define 
\[ \widehat\deg\ \Zed(T) \ := \ \sum_{\lie{p} \subset o_k}  \chi \left( \Zed(T)_{\lie p} , \, \mathcal O_{\Zed(t_1)} \otimes^{\mathbb L} \mathcal O_{\Zed(t_2)} \right)  \ \log(N(\lie{p})) \]
as the sum of the contributions to the Serre intersection multiplicity of $\Zed(t_1)$ and $\Zed(t_2)$ that appear within the support of  $\Zed(T)$, weighted by the factors $\log N(\lie{p})$.

Our main result relates this degree to the $T$'th Fourier coefficient of  the derivative at the centre of symmetry ($s=0$) of an Eisenstein series $ \mathcal E(z, s)$, 
 constructed in  \refSec{EisSect}, on the Hermitian upper half-space $\lie{H}_2$ of genus 2; here the derivative 
 \[ \mathcal E'(z, s) \ = \ \frac{\mathrm d}{\mathrm ds}\, \mathcal E(z,s)\]
  is taken with respect to the variable $s \in \C$. 
\begin{maintheorem*} Suppose $T \in \Herm_2(o_k)$ is positive definite, and define
\[ \Diff(T) \ := \ \{ \ell \ \text{inert}, \  \ell \nmid d  ,  \ \ord_{\ell} \det T  \text{ odd } \} \ \bigcup \ \{ \ell | d \ ,  \ \ord_{\ell} \det T \text{ even} \}. \]
If $|\Diff(T)| \geq 1$, then
\[ \widehat\deg \  \Zed(T) \ q^T \ = \ \frac{2 \, h(k)}{|o_k^{\times}|} \ \mathcal E'_T(z, 0) \]
where  $h(k)$ is the class number of $k$, and for $z \in \lie{H}_2$, we set $q^T := e( tr (T z))$.
\end{maintheorem*}
The novel aspects of this theorem emerge when $\Diff(T) = \{p\}$ is a single inert prime  $p $ dividing $d$. In this case, the cycle $\Zed(T) $ is supported in the fibre $\mathcal M_p$. This fibre in turn bears a close relationship to the \emph{Drinfeld upper half-plane} $\D$ which, as we recall in \S 2, admits an intepretation as a moduli space of $p$-divisible groups. 

Our first task, carried out in \S2, is therefore to consider a family of \emph{local Kudla-Rapoport} divisors on $\D$ defined in terms of deformations of $p$-divisible groups, and study their intersection behaviour. Explicit equations for these divisors were found in \cite{San1}; by combining that information with a combinatorial description of the reduced locus $\D_{red}$ in terms of the Bruhat-Tits building for $SL_2(\Q_p)$, we arrive at an explicit, and surprisingly simple, formula for the intersection number of two local Kudla-Rapoport divisors, see \refCor{mainLocGeomCor}.

In \S 3, we show that one can express the same formula in terms of \emph{local representation densities} and their derivatives, which play an esssential role in the determination of the Fourier coefficients of Eisenstein series. The key tool is the development of a closed-form expression, in the particular case that we need, of Hironaka's general formula \cite{Hiro} for Hermitian representation densities in the  unramified setting,  cf.\ \refProp{RepDensProp}. 

Finally, we connect the local calculations with the global setting and prove the main theorem; our approach here follows \cite[\S7--\S10]{KRunnglob} quite closely.

The first half of \S 4 concerns structural results regarding the geometry of $\mathcal M$ and the special cycles $\Zed(T)$; in particular, a \emph{$p$-adic uniformization} result allows us to express $\widehat\deg \ \Zed(T)$ as a product of a local factor, corresponding to a local intersection number as calculated in \S2, and a global factor that is essentially a lattice point count.

We then turn to calculating the right hand side of our main theorem; after recalling some general facts about Siegel-Weil Eisenstein series and their Fourier coefficients, we describe the particular choice of parameters that give rise to the Eisenstein series that figures in our main theorem, and compute the $T$'th Fourier coefficient of its derivative in \refThm{MainThmEis}. The formula we derive also decomposes as the product of a local factor, expressed in terms of representation densities, with a global lattice point count. A direct comparison of the two formulae yields the proof of the main theorem, cf.\ \refCor{mainCor}. 

\subsection*{Acknowledgements}  This work was supported by the SFB/TR 45 `Periods, Moduli Spaces, and Arithmetic of Algebraic Varieties' of the DFG (German Research Foundation).

\subsection*{Notation} Throughout this paper, $k$ will be a fixed imaginary quadratic field, with ring of integers $o_k$ and discriminant $\Delta <0$. We denote the non-trivial Galois action on $k$ by $a \mapsto a'$. 

Let $\widehat\Z := \prod_{\ell} \Z_{\ell}$ and for any prime $p$, we put $\widehat\Z{}^p = \prod_{\ell \neq p} \Z_{\ell}$. If $M$ is a $\Z$-module, we set 
\[ \widehat M \ : = \ M \otimes_{\Z} \widehat \Z , \qquad \text{and} \qquad \widehat M^p := M \otimes_{\Z} \widehat\Z{}^p. \]
\section{Local Kudla-Rapoport cycles on the Drinfeld upper half-plane} \label{DrinfeldSec}
In this section, we fix an inert prime $p\neq 2$. Let $k_p$ denote the completion, and $o_{k,p} \subset k_p$ the ring of integers. Fix an algebraic closure $\F = \overline{\F_p}$ and an embedding
\[ \tau_0 \colon  \ o_{k,p} \big/ (p)  \ \hookrightarrow \ \F. \]
Denote the nontrivial Galois operator on $k_p$ by $a \mapsto a'$, and let $\tau_1(a) := \tau_0(a')$ be the conjugate embedding;
if $W = W(\F)$ is the ring of Witt vectors, then $\tau_0$ and $\tau_1$ lift uniquely to embeddings
\[ \tau_i \colon \  o_{k,p} \ \to \ W. \]
Finally, we let $\Nilp$ denote the category of $W$-schemes such that $p$ is locally nilpotent, and for $S \in \Nilp$, we set $\overline S \ := \ S \times_{W} \F$.

We begin by recalling the construction of the Drinfeld upper-half plane as a moduli space for $p$-divisible groups, following \cite{KRDrinfeld}.
\begin{definition} Let $S \in \Nilp$. An \emph{almost-principally-polarized  $p$-divisible group} over $S$ is a triple $(X, i , \lambda)$ consisting of 
\begin{compactenum}[(i)]
\item a $p$-divisible group $X$ over $S$ of height 4 and dimension 2;
\item an action $i \colon o_{k,p} \to \End( X)$ satisfying the \emph{signature (1,1) condition}: for every $a \in o_{k,p}$, the characteristic polynomial of $i(a) $ on the Lie algebra $\Lie( X)$ is
\[\det( T - \ i (a)|_{\Lie( X)}) \ = \ (T - a)(T - a') \ \in \ \mathcal O_S[T]; \]
\item and a polarization $\lambda$ such that 
\[ p \cdot \ker(\lambda) = 0, \qquad |\ker(\lambda)| = p^2,   \]
and such that  the induced Rosati involution $*$ satisfies 
\[ i(a) ^* = i(a') \qquad \text{for all } a \in o_{k,p}. \]
\end{compactenum}
\end{definition}

The following lemma can easily be proved by considering the classification of rational Dieudonn\'e modules over algebraically closed fields.
\begin{lemma} Suppose $(\X, i_{\X}, \lambda_{\X})$ is a triple over $\F = \overline{\mathbb{F}_p}$ as above. Then $\X$ is supersingular. Moreover, the data $(\X, i_{\X}, \lambda_{\X})$  is unique up to isogeny; that is, given another  triple $(\X', i_{\X'}, \lambda_{\X'})$, there exists an $o_{k,p}$-linear isogeny $\X \to \X'$ such that the pullback of $\lambda_{\X'}$ is $\lambda_{\X}$.  \qed
\end{lemma}

We fix a triple $(\X, i_{\X}, \lambda_{\X})$ over $\F$ once and for all, which will serve as  a base point for the following moduli problem:

\begin{definition} \label{DrinfeldUHPDef}
Let $\D$ denote the following moduli problem over $\Nilp$: for a base scheme $S\in \Nilp$, the points $\D(S)$ parametrize isomorphism classes of tuples
\[ \D(S) \ := \ \left\{ \underline X \ = \ (X, i_X, \lambda_X, \rho_X ) \right\}_{/ \simeq} \]
where $(X, i_X, \lambda_X)$ is an almost-principally-polarized $p$-divisible group over $S$, and
\[ \rho_{X} \colon X \times_S \overline S \ \to \ \X \times_{\F} \overline S \]
is a height-0 quasi-isogeny of $p$-divisible groups over $\overline S := S \times_W \F$ that is equivariant with respect to the action of $o_k$, and such that $\rho_{X}^*(\lambda_{X,\overline S}) =  \lambda_{\X, \overline S}$. 

This moduli problem is representable by (a formal model of) the Drinfeld upper half-plane, cf.\ \cite{KRDrinfeld}. In particular, it is a regular formal scheme over $\Spf(W)$.
\end{definition}

Let $\Y $ be supersingular $p$-divisible group $\Y$ over $\F$ of dimension $1$ and height $2$ (i.e.\ the $p$-divisible group of a supersingular elliptic curve). We also fix an action $i_{\Y} \colon o_{k,p} \to \End(\Y)$, and a principal polarization $\lambda_{\Y}$ such that the induced Rosati involution acts by Galois conjugation on the image $i_{\Y}(o_{k,p})$.

Following \cite{KRunnloc}, we define the space of \emph{special local homomorphisms}: 
\begin{equation}
\V  \ := \ \Hom_{o_{k,p}}(\Y, \X ) \otimes_{\Z_p} \Q_p
\end{equation}
This space comes equipped with a natural Hermitian form: for $b_1, b_2 \in \V$, put
\[ (b_1, b_2) \ := \ \lambda_{\Y}^{-1} \circ b_2^{\vee} \circ \lambda_{\X} \circ b_1 \ \in \ \End_{o_{k,p}}(\Y) \otimes \Q_p \ \simeq \ k_p. \]
It turns out that with this Hermitian form, $\V$ is split, cf.\ \cite[Remark 3.4]{San1}.

%
%
%

\begin{definition} 
\begin{compactenum}[(i)] 
\item Suppose $\Lambda \subset \V$ is an $o_{k,p}$-lattice, and let $\Lambda^{\sharp} $ denote the dual lattice. We say that $\Lambda$ is a `vertex lattice' of type 0 (resp. type 2) if $  \Lambda^{\sharp}= \Lambda$ (resp. $\Lambda^{\sharp} = p \Lambda$). In the sequel, we shall use the term `lattice' to mean a vertex lattice of type 0 or 2.  
\item Let $\mathscr B$ denote the Bruhat-Tits tree for $SU(\V)$, which is a graph with the following description. The vertices are the vertex lattices, and edges only occur between lattices of differing type. Two lattices $\Lambda$ and $\Lambda'$ of type 0 and 2 respectively are connected by an edge if and only if 
\[ p \Lambda' \ \subset \ \Lambda \ \subset \ \Lambda', \]
where the successive quotients are $\F_{p^2}$-vector spaces of dimension 1. In particular, this graph is a $(p+1)$-regular  tree.
\end{compactenum}
\end{definition}

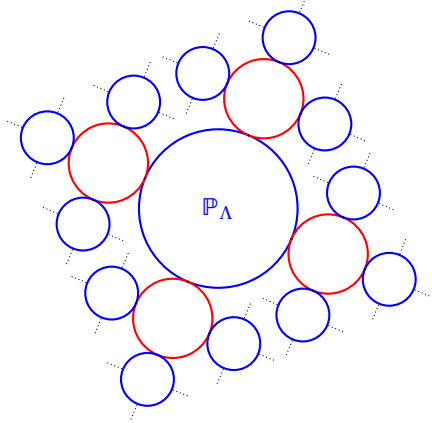
\begin{figure}[h]  

	\begin{tikzpicture}[grow cyclic]
\tikzstyle{level 1}=[level distance = 45pt, sibling angle=90]
\tikzstyle{level 2}=[level distance = 25pt,sibling angle=90]
\tikzstyle{level 3}=[level distance = 20pt, sibling angle=90]
\path[rotate = 22.5]
	node[style={circle, draw, thick, color = blue, minimum size = 60pt, inner sep = 2pt}] {$\mathbb P_{\Lambda}$}
		child foreach \cI in {1,...,4}{ 
			node[style={circle, draw, thick, color = red, minimum size = 30pt, inner sep = 2pt}] {}
			child foreach \cII in {1,...,3} { 
			  	node[style={circle, thick, draw, color = blue, minimum size = 20pt, inner sep = 2pt}] { }
									child[edge from parent/.style={densely dotted, draw}] foreach \cIIII in {1,...,3}{ node{} }
			}
		}  
		;
\end{tikzpicture}
\caption{A portion of  $\D_{red}$ for $p=3$ as a union of projective lines indexed by vertex lattices.}
 \label{DrinRedFig}
\end{figure}

The reduced locus $\D_{red}$ can be described by the Bruhat-Tits tree $\mathscr B$ as follows. Each irreducible component of $\D_{red}$ is a projective line $ \mathbb P_{\Lambda} $ over $\F$ indexed by a vertex lattice. Two such lines $\mathbb P_{\Lambda}$ and $\mathbb P_{\Lambda'}$ intersect at at most one point, which we call a `superspecial' point, and this happens if and only if $\Lambda$ and $\Lambda'$ are neighbours in $\mathscr B$. On a given component $\mathbb P_{\Lambda}$, the superspecial points are precisely the $\F_p$-rational points, of which there are $p+1$, see Figure \ref{DrinRedFig}.

\begin{definition} \label{D0Def}
Let $\D_0$ be  the moduli space on $\Nilp$ that, for a scheme $S \in \Nilp$, parametrizes isomorphism classes of tuples
\[ \D_0(S) \ := \ \left\{  \underline Y =  (Y, i_Y, \lambda_Y, \rho_Y) \right\} /_{ \simeq};\]
here $Y$ is a $p$-divisible group over $S$ of dimension 1 equipped with an action $i_Y \colon o_{k,p} \to \End(Y)$, and $\lambda_Y$ is a compatible principal polarization. Finally, 
\[ \rho_Y \colon \  Y \times_{S}  \overline S \ \to \  \Y \times_{\F} \overline S \]
is an $o_{k,p}$-linear quasi-isogeny of height 0.
\end{definition}

%
%
%

Note that the moduli functor $\D_0$ is trivial, i.e.\ it is represented by $\Spf(W)$. Indeed,  by using Gross' theory of (quasi-)canonical liftings for example, cf.\  \cite{Gross}, one may show that for any $S \in \Nilp$, there is a \emph{unique} lift $\underline \Y_S$ of $\underline \Y$ that is determined by the action $i_{\Y}$. 

We turn now to the local \emph{Kudla-Rapoport} cycles, which are parametrized by elements in $\V$:

\begin{definition} \label{locCyclesDef}
Let $b \in \V$. We define the \emph{local Kudla-Rapoport cycle} Z(b) as the closed formal subscheme of $\D_0 \times_W \D$ defined by the following moduli problem: for $S \in \Nilp$, the set of $S$-points  $Z(b)(S)$ is the locus of pairs $(\underline Y, \underline X )  \in (\D_0 \times_W \D)(S)$ such that the quasi-morphism
\[ \rho_{X}^{-1} \circ b \circ \rho_Y \colon \  Y   \times_{S} {\overline S}  \  \to  \ X \times_S \overline S \]
lifts to a morphism $Y \to X$ over $S$. 

Similarly, if $\mathbf b = [b_1, b_2] \in \V^2$, then we define the cycle $Z(\mathbf b)$ to be the locus $(\underline Y, \underline X)$ where  $\rho_{X}^{-1} \circ b_i \circ \rho_Y$ lifts for $i = 1,2$. 
\end{definition}

These cycles were studied in detail in \cite{San1}, where it was shown that the cycles $Z(b)$ corresponding to a single element $b \in \V$ are in fact divisors. We also have an explicit decomposition of these divisors into irreducible components, as follows; since $\D_0 \simeq \Spf W$, we shall henceforth implicitly identify the formal schemes 
\[ \D_0 \times_{\Spf(W)} \D \ \simeq \ \D\]
over $\Spf(W)$.

\begin{definition} \label{centLattDef}
 Let $b \in \mathbb V$ with $\ord_p (b,b)  = m \geq 0$. Set $t = \lfloor \frac{m+1}{2} \rfloor$ and $\beta := p^{-t} b$, so that $\ord_p (\beta,\beta)$ is either $0$ or $-1$. Then by \cite[Lemma 3.8]{San1}, there exists a unique vertex lattice $\Lambda$ such that $\beta \in \Lambda \setminus p \Lambda$ and morever $\Lambda$ is of type 0 (resp. type 2) if $m$ is even (resp. odd). We call this lattice the \emph{`central lattice'} for $b$. 
\end{definition}

\begin{theorem}[{\cite[Theorem 3.14]{San1}}] \label{San1Thm}
Let $ b \in \mathbb V$, such that $m := \ord_p (b, b) \geq 0$.
Then, as a cycle on $\mathcal D$, 
\[ Z( b) \ = Z( b)^{h} + \sum_{\Lambda \in \mathcal B( b)} m( b, \Lambda) \ \mathbb P_{\Lambda} \ =: \ Z(b)^h + Z(b)^v, \]
where 
\begin{compactenum}[(i)]
\item $Z( b)^{h}$ is a horizontal divisor isomorphic to $\Spf(W)$ that meets the special fibre of $\mathcal D$ at a single non-superspecial point on the component $\mathbb P_{\Lambda_0}$ corresponding to the central lattice $\Lambda_0$;
\item $\mathcal B(b) := \{ \Lambda \text{ vertex lattice} \ | \ b \in \Lambda \} $;
\item and the multiplicities $m( b, \Lambda)$ are given by the  formula
\begin{equation*} 
m( b, \Lambda) \ =  \
\begin{cases}  t  \ - \  \lfloor d(\Lambda, \Lambda_0) / 2 \rfloor ,  & \text{if } m = 2t \text{ is even} \\
t  \ - \ \lfloor (d(\Lambda, \Lambda_0 )+1)/2 \rfloor , & \text{if } m = 2t-1 \text{ is odd};
 \end{cases} 
\end{equation*}
here $d(\Lambda, \Lambda_0$ is the distance between the two vertex lattices in the Bruhat-Tits tree. 
\end{compactenum}  \qed
\end{theorem} 
\begin{remark}  \label{bdyTRmk}
\begin{compactenum}[(i)]
\item For notational consistency, if $b \in \mathbb V$ with $\ord_p (b,b) <0$, we define $Z(b) = 0$ and  $\mathcal B(b) = \emptyset$. 
\item By \cite[Lemma 3.12]{San1},  $\Lambda \in \mathcal B( b)\iff d (\Lambda, \Lambda_0) \leq  m $. In other words, $\mathcal B(b)$  is simply the ball of radius $m$ in the Bruhat-Tits tree, centred at $\Lambda_0$.
\item Note that for vertex lattices that lie on the boundary of $\mathcal B( b)$, i.e.\ those $\Lambda$ for which $d(\Lambda, \Lambda_0) = m$, we have $m( b, \Lambda) = 0$. While such vertices contribute nothing to $Z( b)$, it will make our formulas somewhat neater if include them in $\mathcal B( b)$. 
\item The following reformulation of the multiplicities will also be useful:
\begin{equation} \label{multAltEqn}
  m(b, \Lambda) \ = \  \frac12  \cdot \begin{cases} m-d(\Lambda, \Lambda_0) ,  & \text{if } m \equiv d(\Lambda, \Lambda_0) \pmod 2 \\
 m - d(\Lambda, \Lambda_0) + 1, & \text{if } m \not\equiv d(\Lambda, \Lambda_0) \pmod 2. 
 \end{cases} 
\end{equation}
\end{compactenum}
\end{remark}


\subsection{Local intersection numbers} {\ \\}

\noindent Given two formal closed subschemes $Z$ and $Z'$ of $\D$ such that the sum of their defining ideals is open in $\mathcal O_{\D}$, we define their intersection number to be
\[ \la Z, \ Z' \ra  \ := \ \chi( \mathcal O_Z \otimes^{\mathbb L} \mathcal O_{Z'} ) , \]
where the tensor product is taken in the derived sense, and $\chi$ is the Euler characteristic of the resulting complex cf.\ \cite[\S 4]{KRinv}.

If $Z(\mathbf b)$ is a cycle corresponding to a pair $\mathbf b = [b_1, b_2] \in \V^2$, we set 
\[ \deg \, Z(\mathbf b) \ := \ \la Z(b_1), \, Z(b_2) \ra. \]

The aim of this section is to compute the intersection $\la Z( b_1), Z( b_2) \ra$ of two local cycles attached to  linearly independent vectors $ b_1,  b_2 \in \mathbb V$ where $\ord_p (b_i, b_i) \geq 0$. 
Via \refThm{San1Thm}, the intersection pairing $\la Z( b_1), Z( b_2) \ra$ can be expanded as
\begin{align}  \label{intExpEqn}
\langle Z( b_1), Z( b_2) \rangle \ =&  \ \langle Z( b_1)^{h}, Z( b_2) \rangle  \ + \ \sum_{\Lambda \in \mathcal{B}( b_1) \cap \mathcal{B}( b_2) }  m( b, \Lambda) \ \langle \mathbb{P}_{\Lambda}, Z( b_2) \rangle \\ \notag
=& \ \langle Z( b_1)^{h}, Z( b_2)^h \rangle \ + \ \langle Z( b_1)^{h}, Z( b_2)^v \rangle   + \sum_{\Lambda \in \mathcal{B}(b_1) \cap \mathcal{B}( b_2) }  m( b_1, \Lambda) \ \langle \mathbb{P}_{\Lambda}, Z( b_2) \rangle.
\end{align}

\begin{lemma} \label{P1IntLemma}
 Let $\Lambda_1$ and $\Lambda_2$ denote the `central lattices' for $ b_1$ and $ b_2$ respectively, in the notation of \refThm{San1Thm}. Then 
\[ \la Z( b_1)^{h} ,  \ Z( b_2)^v \ra  \ = \ \begin{cases} m( b_2, \Lambda_1), & \text{if } \Lambda_1 \in \mathcal B( b_2), \\ 0, & \text{otherwise}. \end{cases} \]
For any vertex lattice $\Lambda$, we also have
\[ \la \mathbb P_{\Lambda}, \ Z( b_2) \ra \ =  \
\begin{cases}
	1, & 	\text{if } \Lambda \in \mathcal B( b_2) \text{ and } d(\Lambda, \Lambda_2)  \equiv \ord_p (b_2, b_2) \pmod{2}  \\ 
	-p, & 	\text{if } \Lambda \in \mathcal B( b_2) \text{ and } d(\Lambda, \Lambda_2)  \not\equiv \ord_p( b_2, b_2) \pmod{2} \\
	0 & 	\text{if } \Lambda \notin \mathcal B( b_2).
\end{cases}
\]

\begin{proof}
By \cite[Equation (4.7)]{KRinv}, we have that for any vertex lattice $\Lambda$,
\[ \la Z( b_1)^{h},\,  \mathbb P_{\Lambda} \ra = \begin{cases}1, & \text{if } \Lambda = \Lambda_1 \\ 0, & \text{otherwise}. \end{cases} \]
The first statement in the proposition follows immediately.

Next, we recall the following formula, cf.\  \cite[Lemma 4.7]{KRinv}: if $\Lambda$ and $\Lambda'$ are any two vertex lattices, then
\[ \la \mathbb P_{\Lambda}, \mathbb P_{ \Lambda' } \ra \ = \ \begin{cases} -(p+1), & \text{if } \Lambda = \Lambda' \\ 1, & \text{if } d(\Lambda, \Lambda') = 1 \\ 0, & \text{otherwise}. \end{cases} \]
Now suppose $\Lambda$ is any  vertex lattice, so that
\[ \la \mathbb P_{\Lambda}, \ Z( b_2) \ra = \la \mathbb P_{\Lambda}, \ Z( b_2)^{h} \ra \ + \sum_{\substack{ \Lambda' \in \mathcal{B}( b_2) \\ d(\Lambda', \Lambda) \leq 1 }} m( b, \Lambda') \  \la \mathbb P_{\Lambda} , \ \mathbb P_{\Lambda'} \ra.  \]
When $\Lambda \notin \mathcal{B}( b_2)$, it follows immediately that $\la \mathbb P_{\Lambda}, Z( b_2) \ra =0$, cf. \refRemark{bdyTRmk} for the case when $\Lambda$ is of distance 1 from the boundary.

Suppose next that $\Lambda \in \mathcal{B}( b_2)$ and $\Lambda \neq \Lambda_2$. Then $\Lambda$ has one neighbour, say $\Lambda^{\sharp}$, that is strictly closer to $\Lambda_2$ than $\Lambda$ is. Suppose further that $d(\Lambda, \Lambda_2) \equiv m \pmod{2}$. Then \eqref{multAltEqn} implies that 
\[ m( b_2, \Lambda^{\sharp}) = m( b_2, \Lambda ) + 1.\]
For any other neighbour $\Lambda^{\flat}$ of $\Lambda$, of which there are $p$,
 \[ m( b_2, \Lambda^{\flat} ) = m( b_2, \Lambda). \]
Therefore, we obtain
\begin{align*}
 \la \mathbb P_{\Lambda}, \ Z( b_2) \ra \ &= \ m( b_2, \Lambda^{\#}) \la  \mathbb P_{\Lambda},  \mathbb P_{\Lambda^{\sharp }}\ra +   m( b_2, \Lambda) \la  \mathbb P_{\Lambda},  \mathbb P_{\Lambda}\ra \ + \ \sum_{\Lambda^{\flat}} \ m( b_2, \Lambda^{\flat}) \la  \mathbb P_{\Lambda},  \mathbb P_{\Lambda^{\flat }}\ra 
\\
&=\  \left(  m( b_2, \Lambda) + 1 \right) \ + \ m( b_2, \Lambda)( - p - 1) \ + \ p\ m( b_2, \Lambda)  \\
&= \ 1.
\end{align*}
The case where $d(\Lambda, \Lambda_2) \not\equiv m \pmod 2 $ follows from similar considerations. 

Finally, suppose $\Lambda = \Lambda_2$. If $\ord_p (b_2, b_2) = 2t $ is even, then we see by \eqref{multAltEqn} that $\Lambda_2$ and all of its neighbours occur in $Z( b_2)$ with the same multiplicity $t$, so that
\begin{align*}
  \la \mathbb P_{\Lambda_2}, \ Z( b_2) \ra \ &= \   \la \mathbb P_{\Lambda_2}, \ Z( b_2)^{h} \ra  \ + \ t \cdot  \la \mathbb P_{\Lambda_2}, \mathbb P_{\Lambda_2} \ra   \ + \sum_{d(\Lambda^{\flat}, \Lambda_2) = 1}  t  \la \mathbb P_{\Lambda_2}, \mathbb P_{\Lambda^{\flat}} \ra  \\ 
 &= \  1 \ + \ t(-p-1) \ + \ t \cdot (p+1) \ = \ 1.
\end{align*}
On the other hand, if $\ord_p(b_2,b_2) = 2t -1 $ is odd, then 
\[ m(b_2,\Lambda_2) \ = \  t \ =  \ m(b_2, \Lambda^{\flat}) + 1\]
  for every neighbour $\Lambda^{\flat}$. Thus
\begin{align*}
  \la \mathbb P_{\Lambda_2}, \ Z( b_2) \ra &=   \la \mathbb P_{\Lambda_2}, \ Z( b_2)^{h} \ra  + t \cdot  \la \mathbb P_{\Lambda_2}, \mathbb P_{\Lambda_2} \ra   + \sum_{d(\Lambda^{\flat}, \Lambda_2) = 1}  (t-1)  \la \mathbb P_{\Lambda_2}, \mathbb P_{\Lambda^{\flat}} \ra  \\ 
 &= 1 \ + \ t(-p-1) \ + \ (t-1) \cdot (p+1) = -p. 
\end{align*}
\end{proof}
\end{lemma}
\begin{lemma}  \label{TIntersectLemma}
Let $ b_1,  b_2 \in \mathbb V$, with corresponding `central lattices' $\Lambda_1$ and $\Lambda_2$  respectively and assume 
\[ m_1 \ := \  \ord_p(b_1, b_1) \ \leq  \ \ord_p(b_2,b_2) \ =: \ m_2 . \]
 Suppose further that $\mathcal{B}( b_1) \cap \mathcal B( b_2)$ is non-empty, and let $\mathfrak{A}$ denote the unique shortest path between $\Lambda_1$ and $\Lambda_2$. Then  the intersection $\mathcal{B}( b_1) \cap \mathcal B( b_2)$ is the ball of radius
\begin{equation} \label{radiusDefEqn}
 r \ :=  \ \min \left( \frac{ m_1+ m_2 - d(\Lambda_1, \Lambda_2)}{2}, \, m_1 \right)
\end{equation}
around the unique vertex lattice $\Gamma\in \mathfrak{A}$ such that
\begin{equation} \label{centLattDistEqn}
 d(\Lambda_1, \Gamma) = m_1-r, \quad  \text{and} \quad d(\Lambda_2, \Gamma) = d(\Lambda_1, \Lambda_2) - (m_1-r) .
\end{equation}
\begin{proof} This follows easily from the fact that the $\mathscr B$ is a tree, and that $\mathcal B(b_1)$ and $\mathcal B(b_2)$ are balls of radius $m_1$ and $m_2$ respectively.
\end{proof}
\end{lemma}
\begin{remark} 
Recall the \emph{type} of the central lattice $\Lambda$ attached to $b \in \V$ is determined by the parity of $\ord_p(b, b)$, as in  \refDef{centLattDef}. As lattices of differing type are always at an odd distance apart, 
\begin{equation}  \label{parityEqn}
 m_1 \ +  \ m_2 \ \equiv \ d(\Lambda_1, \Lambda_2) \pmod 2,
\end{equation}
 and so $r \in \Z$. $\diamond$
\end{remark}
%
%
%
%
%
We now come to the main theorem of this section:
\begin{theorem} \label{localDegThm}
Suppose $ b_1,  b_2 \in \mathbb V$ are linearly independent vectors, with $m_1:= \ord_p(b_1, b_1)$ and $m_2 := \ord_p(b_2, b_2)$ and central lattices  $\Lambda_1$ and $\Lambda_2$ respectively. Assume that $0 \leq m_1 \leq m_2$. 
\begin{compactenum}[(i)]
\item If $\mathcal{B}( b_1) \cap \mathcal{B}( b_2) = \emptyset$, then $ \la Z( b_1), Z( b_2) \ra = 0$.
\item If $\mathcal{B}( b_1) \subset \mathcal{B}( b_2)$, then
\[  \la Z( b_1), \, Z( b_2) \ra \ = \  \frac{m_1 + m_2 - d(\Lambda_1, \Lambda_2) }{2}  -  \ p\left( \frac{p^{m_1} - 1}{p-1} \right)   \ +\ \la Z(b_1)^h,\, Z(b_2)^h \ra. \]
(See \refProp{hhProp} below for the calculation of the final `h-h' term.)
\item Suppose $\mathcal{B}( b_1) \cap \mathcal{B}( b_2) \neq \emptyset$ but $\mathcal{B}( b_1) \not\subset \mathcal{B}(b_2)$, and let $r$ be as in \eqref{radiusDefEqn}. Then 
\[ \la Z( b_1), Z( b_2) \ra \ = \ r  -  \ p\left( \frac{p^r - 1}{p-1} \right). \]
	\end{compactenum}
\begin{proof}
(i) When $\mathcal{B}( b_1) \cap \mathcal{B}(b_2) = \emptyset$, one sees immediately from \eqref{intExpEqn} that the intersection pairing vanishes.  
\\

\noindent \fbox{Case (ii): $\mathcal{B}(b_1) \subset \mathcal{B}(b_2)$} \\

In this case, we have $\mathcal{B}(b_1) \cap \mathcal{B}(b_2) = \mathcal{B}(b_1)$, which is the ball of radius $m_1$ centred at $\Lambda_1$.  Let
\[ Z(b_1)^{v}  \ =  \  \sum_{\Lambda \in \mathcal{B}(b_1)} \ m(b_1, \Lambda)  \ \mathbb P_{\Lambda} \]
denote the 'vertical' part of $ Z(b_1)$. 
We may express its contribution to the intersection by expanding radially from $\Lambda_1$:
\begin{align*}
 \la Z(b_1)^{v} , \ Z(b_2) \ra \ = \  m(b_1, & \Lambda_1) \cdot \la \mathbb P_{\Lambda_1},  Z(b_2) \ra \ + \ \sum_{\substack{ \Lambda \\ d(\Lambda_1, \Lambda) = 1}} m(b_1, \Lambda) \cdot \la \mathbb P_{\Lambda}, Z(b_2) \ra \\
 &+ \cdots + \sum_{\substack{ \Lambda \\ d(\Lambda_1, \Lambda) = m_1}} m(b_1, \Lambda) \cdot \la \mathbb P_{\Lambda}, Z(b_2) \ra
\end{align*}
Suppose first that $m_1$ is even. Then
\[ m(b_1,\Lambda) \ =  \ \frac{m_1}{2}  \ -\ \lfloor d(\Lambda_1, \Lambda)/2\rfloor \]
and by \eqref{parityEqn}, we have that $d(\Lambda_1, \Lambda_2) \equiv m_2 \pmod 2$. Therefore, applying \eqref{multAltEqn}, 
\begin{align*}
 \la Z(b_1)^{v} , Z(b_2) \ra \ &= \ \frac{m_1}2 \cdot (1) \ + \ (p+1) \frac{m_1}2 \cdot (-p)  \ + \ p (p+1) \left( \frac{m_1}{2} -1 \right) \cdot (1) \\
 & \qquad + \cdots + p^{m_1-2}(p+1) \left( \frac{m_1}{2} - [(m_1-1)/2] \right) (-p) + p^{m_1-1}(p+1) \left( \frac{m_1}{2} - [m_1/2] \right)  \\
 &= \ \frac{m_1}{2} \ + \ (p+1)( -p -p^3 - \cdots -p^{m_1-1} ) \\
 &= \frac{m_1}{2} \ - \ p\left( \frac{p^{m_1}- 1}{p -1} \right)
 \end{align*}
On the other hand, using \eqref{multAltEqn} and \refLemma{P1IntLemma}, 
\[ \la Z(b_1)^{h}, \ Z(b_2)^v \ra \ = \ m(b_2, \Lambda_1) \ = \ \frac{m_2 - d(\Lambda_1,\Lambda_2)}{2}, \]
and so 
\begin{align*}
\la Z(b_1), \ Z(b_2) \ra \ =&  \ \la Z(b_1)^h, Z(b_2)^h \ra \ + \ \la Z(b_1)^h, Z(b_2)^v \ra \ + \  \la Z(b_1)^v, Z(b_2) \ra \\
=& \ \frac{m_1 + m_2 - d(\Lambda_1, \Lambda_2)}{2} \ - \  p \left( \frac{p^{m_1} - 1}{p-1} \right) \ + \  \la Z(b_1)^{h}, Z(b_2)^h \ra  
\end{align*}
as required. 

Next, if $m_1$ is odd, then setting $t_1 = (m_1+1)/2$ gives
\[ m(b_1, \Lambda)\ = \  t_1  \ - \ \left\lfloor \frac{d(\Lambda_1, \Lambda) +1 }{2} \right\rfloor, \qquad \text{for any } \Lambda \in \mathcal B(b_1) . \]
Since $d(\Lambda_1, \Lambda_2)  \not\equiv m_2 \pmod 2$, 
\begin{align*}
\la Z(b_1)^{v} , & Z(b_2) \ra = m(b_1,  \Lambda_1) \cdot \la \mathbb P_{\Lambda_1},  Z(b_2) \ra \ + \ \sum_{\substack{ \Lambda \\ d(\Lambda_1, \Lambda) = 1}} m(b_1, \Lambda) \cdot \la \mathbb P_{\Lambda}, Z(b_2) \ra \\
 &\qquad + \cdots + \sum_{\substack{ \Lambda \\ d(\Lambda_1, \Lambda) = m_1-2}} m(b_1, \Lambda) \cdot \la \mathbb P_{\Lambda}, Z(b_2) \ra \\
&=\  t_1\cdot (-p) + (p+1)(t_1-1) \cdot (1) + p(p+1)(t_1-1) \cdot (-p) \\
& \qquad + \cdots + p^{m_1-3}(p+1)(t_1- [(m_1-1)/2]) + p^{m_1-2}(p+1)(t_1 - [m_1/2])(-p) \\
&= \ t_1 \ - \ (p+1) ( 1+ p^2 + p^4 + \cdots + p^{m_1-1} ) \\
&= \ \frac{m_1+1}{2} \ - \ \left( \frac{p^{m_1+1} - 1}{p - 1} \right).
\end{align*}
By \eqref{multAltEqn}, we have
\[ \la Z(b_1)^{h}, \ Z(b_2)^v \ra \ =  \ m(b_2, \Lambda_1) \ =  \ \frac{m_2 - d(\Lambda_1,\Lambda_2)+ 1}{2}, \]
so that 
\begin{align*}
 \la Z(b_1), \ Z(b_2) \ra \ &=  \ \frac{m_1+m_2 - d(\Lambda_1,\Lambda_2)}{2}\ + \ 1  \ -\   \left( \frac{p^{m_1+1} - 1}{p - 1} \right) \\
 &= \ \frac{m_1+m_2 - d(\Lambda_1,\Lambda_2)}{2}  \ - \ p \left( \frac{p^{m_1} - 1}{p-1} \right) \ + \  \la Z(b_1)^{h}, Z(b_2)^h \ra
\end{align*}
as required. 
\\

\noindent \fbox{Case (iii): $\mathcal{B}(b_1) \not\subset \mathcal{B}(b_2)$} 

Recall that the intersection $\mathcal{B}(b_1) \cap \mathcal{B}(b_2)$ is a ball of radius 
\[ r \ = \ \frac{m_1 + m_2 - d(\Lambda_1, \Lambda_2) }{2} , \]
 centred at a vertex $\Gamma$ along the geodesic $\mathfrak{A}$ connecting $\Lambda_1$ and $\Lambda_2$.
 We start by calculating 
\[ ( Z( b_1)^{v}, Z( b_2) ) \] 
as follows. Consider taking a walk along $\mathfrak{A}$, starting from $\Gamma$ towards $\Lambda_1$. We stop walking either after a distance of $r$ units, or when we arrive at $\Lambda_1$, whichever comes first, i.e., we travel a distance of $\min(r, d(\Lambda_1, \Gamma))$. For each vertex we encounter, say after $k$ steps, we add up the contributions of all the lattices branching off of that vertex \emph{away from} $\mathfrak{A}$, and call that sum $F(k)$. More precisely, let $\Gamma^{(k)}$ denote the vertex in $\lie{A}$ which is a distance of $d(\Lambda_1, \Gamma) - k$ away from $\Lambda_1$ and a distance of $k$ away from $\Gamma$. Then the sum
\[ F(k) = \sum_{\Lambda \in \mathcal{F}(k)} m( b_1, \Lambda) \ \langle \mathbb P_{\Lambda}, Z( b_2) \rangle \]
is taken over the set $\mathcal{F}(k)$ consisting of  lattices $\Lambda \in \mathcal B( b_1) \cap \mathcal B (b_2)$ such that the unique shortest path between $\Lambda$ and $\Lambda_1$ first meets $\mathfrak{A}$ at $\Gamma^{(k)}$. 
The points $\Gamma = \Gamma^{(0)}$ and $\Lambda_1$  (if indeed we end up walking that far, as the vertex $\Lambda_1$ contributes if and only if $ d(\Lambda_1, \Gamma) \leq r $) are different from the rest, because at these points there are $p$ directions leading \emph{away} from the path $\mathfrak{A}$, whereas at all the intermediate vertices there are $p-1$. 

For convenience, set $\mu_k := m(b_1, \Gamma^{(k)})$ and $d_k = d(\Lambda_1, \Gamma^{(k)})$. Our first step is to prove the following calculation for $F(0)$:
\begin{leftbar}[0.85\linewidth]
\begin{flalign*}
& \ \ F(0) \ = \ 
	\begin{cases} 
		\mu_0 - p^2 \left( \frac{p^r - 1}{p^2-1}\right), & \text{if }  r \text{ is even,} \\
		- p \left( \frac{p^{r+1} - 1}{p^2-1}\right), & \text{if } r \text{ is odd}. 
\end{cases} &
\end{flalign*}
\end{leftbar}
To prove this, first suppose $r$ is even. Then $d_0 = m_1 - r$ has the same parity as $m_1$. By  \eqref{multAltEqn},  a vertex $\Lambda$ contributing in $F(0)$ that is $s$  units away from $\Gamma$ appears in $Z(b_1)$ with multiplicity $m(b_1, \Lambda) = \mu_0 - \lfloor s/2 \rfloor$. 
Therefore, 
\begin{align*} 
F(0) \ = \ \mu_0 \cdot \la \mathbb P_{ \Gamma}, Z(b_2) \ra + \sum_{\substack{ \Lambda \in \mathcal{F}(0)  \\ d( \Lambda, \Gamma) = 1 }}  & \mu_0 \cdot \la \mathbb P_{\Lambda}, Z(b_2) \ra +  \sum_{\substack{ \Lambda \in \mathcal{F}(0)  \\ d( \Lambda, \Gamma) = 2 }} (\mu_0 -1 ) \cdot \la \mathbb P_{\Lambda}, Z(b_2) \ra  +  \\
&+ \dots +  \sum_{\substack{ \Lambda \in \mathcal{F}(0)  \\ d( \Lambda, \Gamma) = r }} (\mu_0 -(r/2) ) \cdot \la \mathbb P_{\Lambda}, Z(b_2) \ra.
\end{align*}
From \eqref{centLattDistEqn} and \eqref{parityEqn}, it follows that $d(\Lambda_2, \Gamma^{(0)}) \equiv m_2 \pmod 2$, and so applying \refLemma{P1IntLemma}, 
\begin{align*}
F(0) \ &= \ \mu_0\cdot(1) + \sum_{\substack{ \Lambda \in \mathcal{F}(0)  \\ d( \Lambda, \Gamma) = 1 }}   \mu_0 \cdot (-p) +  \sum_{\substack{ \Lambda \in \mathcal{F}(0)  \\ d( \Lambda, \Gamma) = 2 }} (\mu_0 -1 )\cdot(1)  + \dots +  \sum_{\substack{ \Lambda \in \mathcal{F}(0)  \\ d( \Lambda, \Gamma) = r }} (\mu_0 -(r/2) ) \cdot (1) \\ 
&= \  \mu_0 + p \cdot \mu_0 (-p)  + p^2 \cdot(\mu_0 - 1) + p^3\cdot (\mu_0 - 1) (-p) + \dots + p^r \cdot (\mu_0 - r/2)  \\
&= \ \mu_0 - p^2 - p^4 \dots - p^r \\
&= \ \mu_0 - p^2 \left( \frac{p^r - 1}{p^2 - 1} \right),
\end{align*}
as required. 

Similarly, when $r$ is odd,  
\begin{align*} 
F(0) \ &= \ \mu_0 \cdot (-p) \ + \ \sum_{\substack{ \Lambda \in \mathcal{F}(0)  \\ d( \Lambda, \Gamma) = 1 }}(\mu_0 - 1) \cdot (1) + \sum_{\substack{ \Lambda \in \mathcal{F}(0)  \\ d( \Lambda, \Gamma) = 2 }} (\mu_0 -1 ) (-p) +  \cdots \sum_{\substack{ \Lambda \in \mathcal{F}(0)  \\ d( \Lambda, \Gamma) =r }} (\mu_0 - (r+1)/2) \cdot (1) \\
&= \ \mu_0 \cdot (-p)  \ +\   p (\mu_0 - 1)  \ + \  p^2(\mu_0 - 1) \cdot (-p) \ + \ p^3(\mu_0 - 2) \ + \cdots + \ p^{r}(\mu_0 - (r+1)/2) \\
&= -p -p^3 - \cdots p^r \\
&= -p \left( \frac{p^{r+1} - 1}{p^2 - 1} \right), 
\end{align*}
as required. 

Next, when $0 < k  <d_0 = d(\Lambda_1, \Gamma^{(0)})$ and $k \leq r$, we have the following calculation:
\begin{leftbar}[0.85\linewidth]
\begin{flalign*}
& \ \ F(k) \ = \ 
	\begin{cases} 
		\mu_k - (p-1)(p) \left( \frac{p^{r-k} - 1}{p^2-1}\right), & \text{if }  r  \equiv k \pmod 2 \\
		- \mu_k - (p-1) \left( \frac{p^{r- k +1 } - 1}{p^2 - 1} \right) , & \text{if }  r \not\equiv k \pmod 2
	\end{cases} &
\end{flalign*}
\end{leftbar}
To prove this claim, we first consider the case $r \equiv k \pmod 2$. Then 
\[ d(\Lambda_1, \Gamma^{(k)}) \ = \   d(\Lambda_1, \Gamma^{(0)}) - k \  = \ m_1 - r - k  \ \equiv \ m_1 \pmod 2, \]
and 
\[ d(\Lambda_2, \Gamma^{(k)}) \ = \   d(\Lambda_2, \Gamma^{(0)}) + k \ \equiv \  m_2 - r + k  \ \equiv \ m_2 \pmod 2 \]
as well.  Hence, for $\mu_k = m(b_1, \Gamma^{(k)})$, we have
\begin{align*}
F(k) \ &= \ \mu_k\cdot(1) + \sum_{\substack{ \Lambda \in \mathcal{F}(k)  \\ d( \Lambda, \Gamma) = 1 }}   \mu_k \cdot (-p) +  \sum_{\substack{ \Lambda \in \mathcal{F}(k)  \\ d( \Lambda, \Gamma) = 2 }} (\mu_k -1 )\cdot(1)  + \dots +  \sum_{\substack{ \Lambda \in \mathcal{F}(k)  \\ d( \Lambda, \Gamma) = r - k }} \left(\mu_k - \frac{r-k}{2} \right) \cdot (1) \\
&= \mu_k    +  (p-1)  \mu_k \cdot (-p)  + p(p-1) (\mu_k - 1)   + p^2(p-1) (\mu_k - 1)\cdot (-p)  +  \cdots  +  p^{r-k-1}(p-1) \left(\mu_k - \frac{r-k}{2} \right) \cdot (1)  \\
&= \mu_k  \ + \  (p-1)(-p - p^3 - \cdots - p^{r-k-1}) \\
&= \mu_k \ - \ p (p-1) \frac{p^{r-k} - 1}{p^2 - 1} .
\end{align*}
The case when $r \not\equiv k \pmod 2$ follows from similar considerations, and we omit the proof. 

Finally, we consider the case $k = d_0 = d(\Lambda_1, \Gamma^{(0)})$, which only arises when $r \geq d_0$. The calculation is almost identical to the case $k = 0$, and so we omit it; the result is:
\begin{leftbar}[0.85\linewidth]
\begin{flalign*}
& \ \ F(d_0) \ = \ 
	\begin{cases} 
		\mu_{d_0} - p^2 \left( \frac{p^{r- d_0} - 1}{p^2-1}\right), & \text{if }  r \equiv d_0 \pmod 2 \\
		- p \left( \frac{p^{r-d_0+1} - 1}{p^2-1}\right), & \text{if } r \not\equiv d_0 \pmod 2 
\end{cases} &
\end{flalign*}
\end{leftbar}

Now, we consider the sum over all  the contributions $F(k)$. Observe that if $k > 0$ and $k \equiv r \pmod 2$, then
\begin{equation} \label{FkPairSum}
F(k) \ + \ F(k+1) \ = \  \mu_k \  - \  \mu_{k+1} \   - \  \left( p^{r-k} - 1 \right) 
\end{equation}

 We first assume that $r < d_0$ and  $r$ is even. The first assumption implies that $\Lambda_1 \notin \mathcal B(b_1) \cap \mathcal B(b_2)$, and so $\la Z(b_1)^h, Z(b_2) \ra  = 0$ in this case. Thus 
\begin{align*}
\la Z(b_1), \  Z(b_2) \ra  \ &=  \ \la Z(b_1)^{v}, Z(b_2) \ra \   = \  \sum_{k=0}^{r} F(k)  \\
&=  \ F(0) \ + \  F(1) \ +  \ \left( \sum_{n=1}^{r/2 - 1} \ F(2n) \ +  \ F(2n+1)  \right) \ +  \ F(r) \\
&=  \ \sum_{k=0}^r (-1)^k \mu_k  \  - \ (p^2 + p - 1) \  \left( \frac{p^r -1}{p^2 - 1} \right)   - \sum_{n=1}^{r/2}(p^{r-2n} - 1) \\
&=\  \sum_{k=0}^r (-1)^k \mu_k - (p^2 + p) \left( \frac{p^r -1}{p^2 - 1} \right) + \frac{r}2.
\end{align*}
Since $r$ is even, we have $d_0 \equiv m_1 \pmod 2$, and so $\mu_k = \mu_0 + \lfloor (k+1)/2\rfloor$. Thus
\begin{align*} \sum_{k=0}^r (-1)^k \mu_k \  = \ \mu_0 - (\mu_0 + 1) + (\mu_0 + 1)  -  (\mu_0 + 2) \dots + (\mu_0 + r/2)  - (\mu_0  + r/2) = \mu_0
\end{align*}
and recalling \eqref{centLattDistEqn} and \eqref{multAltEqn}, we have $\mu_0 = r/2$. 
Therefore
\[ \la Z(b_1), Z(b_2) \ra \ = \ r - p \left( \frac{p^r - 1}{p-1} \right),  \]
which proves the proposition for $r < d_0$ with $r$ even. A similar calculation gives the same result when $r<d_0$ and $r$ is odd. 

Next, we consider the case that $r \geq d_0$. Then we have
\[ \la Z(b_1), Z(b_2) \ra \ =  \ \la Z(b_1)^{h}, Z(b_2) \ra \ + \ \la Z(b_1)^{v} , Z(b_2) \ra.\]
The proof proceeds by a further  case-by-case analysis, depending on the parity of  $r$ and $d_0$. Suppose that both are even. Then
\begin{align*}  \la Z(b_1)^{v} , Z(b_2) \ra &= F(0) + F(1) + \sum_{n= 1}^{d_0/2 - 1} F(2n) + F(2n+1)  + F(d_0) \\
 &= \mu_0 - p^2 \left(\frac{p^r - 1}{p^2-1} \right) - \mu_1 - (p-1) \left(\frac{p^r - 1}{p^2-1} \right) + \sum_{n=1}^{d_0/2 - 1} \mu_{2n} - \mu_{2n + 1 } - (p^{r - 2n} - 1)  \\
 & \qquad \qquad + \mu_{d_0} - p^2\left( \frac{p^{r - d_0} -1 }{p^2 - 1} \right)  \\
 &= \frac{d_0}{2} - 1 + \sum_{k = 0}^{d_0} (-1)^k \mu_k - (p^2 + p - 1) \left(\frac{p^r - 1}{p^2-1} \right) - \frac{p^r - p^{r - d_0 + 2}}{p^2 - 1} - p^2 \left(\frac{p^{r-d_0} - 1}{p^2-1} \right).
\end{align*}
Using the fact that in this case $\sum_k (-1)^k \mu_k  = \mu_0 = r/2$ and simplifying the above expression, we obtain
\[ \la Z(b_1)^{v} , Z(b_2) \ra = \frac{d_0}{2} + \frac{r}{2} - p \left( \frac{p^r - 1}{p^2 -1 } \right) = \frac{m_1}{2} - p \left( \frac{p^r - 1}{p^2 -1 } \right) .\]
On the other hand, we note that 
\[ \la Z(b_1)^{h} , Z(b_2) \ra = m(b_2, \Lambda_0). \]
Combining \eqref{centLattDistEqn} with the assumptions that $r$ and $d_0$ are even, we have $d(\Lambda_1, \Lambda_2) \equiv m_2 \pmod 2$, and so by \eqref{multAltEqn},
\[ m(b_2, \Lambda_1) = \frac{m_2 - d(\Lambda_1, \Lambda_2)}{2}  =  r - \frac{m_1}{2}. \]
Therefore
\[ \la Z(b_1) , Z(b_2) \ra = r - p \left( \frac{p^r - 1}{p - 1} \right) \] 
as desired, in the case $r \equiv d_0 \equiv 0 \pmod 2$. Again, the remaining cases are entirely analogous, and so we omit the proofs. 
\\

\end{proof}
\end{theorem}
%
%
 
%
We  turn to the `horizontal-horizontal' terms appearing in \refThm{localDegThm}\emph{(ii)}. As usual, let $b_1, b_2 \in \mathbb V$ be linearly independent, with
\[ m_1 \ = \ \ord_p (b_1, b_1) \ \leq \ m_2 :=  \ord_p (b_2, b_2) \] 
and central lattices $\Lambda_1$ and $\Lambda_2$ respectively. For $i = 1,2$, set 
\[ t_i := \lfloor \frac{m_i+1}{2} \rfloor,  \qquad \text{and} \qquad \beta_i \ := \ p^{-t_i} b_i \]
so that $\beta_i \in \Lambda_i \setminus p \Lambda_i$ by the property characterizing central lattices.

\begin{proposition} \label{hhProp}
With notation as in the previous paragraph,
\[ \la Z(b_1)^h, Z(b_2)^h \ra \ = \ \begin{cases} 0 , & \text{ if } \Lambda_1 \neq \Lambda_2 \\ \ord_p (\beta_2, \beta_1'), & \text{ if } \Lambda_1 = \Lambda_2 \end{cases} \]
where $\beta_1' \in \Lambda_1$ is any vector such that $\left\{ \beta_1, \beta_1' \right\}$ form an orthogonal basis for $\Lambda_1$. 

\begin{proof} Recall from \refThm{San1Thm} that $Z(b_i)^h$ meets the special fibre at a single non-superspecial point in the component $\mathbb P_{\Lambda_i}$. Hence if $\Lambda_1 \neq \Lambda_2$, the pairing clearly vanishes. 

Thus we may assume  $\Lambda_1 = \Lambda_2 = \Lambda$, and we suppose further that $\Lambda$ is self-dual, so that $m_1 = 2t_1$ and $m_2 = 2t_2$ are even. We may fix a basis $\{ e, f \}$ for $\Lambda$ with respect to which the Hermitian form is 
\[ h \ \sim \ \begin{pmatrix} & \delta \\ -\delta & \end{pmatrix} \]
where $\delta \in o_{k,p}^{\times}$ is a generator for $k_p  /  \Q_p$ with $\delta' = - \delta$. Writing 
\[ \beta_1 = p^{-t_1} b_1  = r_1 e + s_1 f, \qquad \beta_2 = p^{-t_2} b_2  = r_2 e + s_2 f  \]
we note that 
\[  (\beta_i, \beta_i )  \ = \ \delta \left( r_i s_i' - r_i' s_i \right) \ \in \ \Z_p^{\times}   \]
and so $r_i$, $s_i$, and $r_i s_i' - r_i' s_i $ are all units in $o_{k,p}$. 

Set $\beta_1' = r_1' e + s_1' f$, so that $\{ \beta_1, \beta_1' \}$ form an orthogonal basis for $\Lambda$. Then we may write
\[ \beta_2 \ = \ \mu \, \beta_1 \ + \  \epsilon \, \beta_1', \qquad \text{for some } \mu,  \epsilon \in o_k, \ \  \epsilon \neq 0.  \]
The local equations for these cycles are described explicitly in \cite[\S 3]{San1}; in particular, by Proposition 3.7 there,
\begin{align*}
Z(b_1)^{h}(\F) = Z(b_2)^{h}(\F) \ &\iff  \ \beta_1  \text{ and } \beta_2 \text{ are collinear mod } p \Lambda \\
&\iff \ \mu \in o_k^{\times} \text{ and } \epsilon \in (p).
\end{align*}  
If $\epsilon \in o_{k,p}^{\times}$, then $Z(b_1)^h$ and $Z(b_2)^h$ do not intersect, and so
\[ \la Z(b_1)^h, Z(b_2)^h \ra \ =  \ 0 \ = \ \ord_p \epsilon \ = \ \ord_p (\beta_2, \beta_1') \] 
as required.

Finally, we consider the situation $Z(b_1)^h(\F) = \{x \} =  Z(b_2)^h(\F)$. As described in \cite[Proposition 3.10]{San1}, the point $x$ has a formal affine neighbourhood
\begin{equation}
 \Spf \, W[T, (T^p - T)^{-1}]^{\vee}  \ =: \ \Spf \, R
\end{equation}
such that   the two cycles $Z(b_1)^h$ and $Z(b_2)^h$ are given by 
\[ \Spf \, R / (r_1 T - s_1) \qquad \text{and} \qquad \Spf \, R / (r_2 T - s_2)  \]
respectively. As $\epsilon \in o_{k,p}^{\times}$, we may write
\[ r_2 T - s_2 \ = \ \mu \cdot \left( r_1 T \ -\  s_1  \ + \  \epsilon \mu^{-1} ( r_1' T - s_1') \right),  \]
and since the cycles intersect properly, it follows immediately that
\begin{align*}
 \chi( \mathcal O_{Z(b_1)^h} \otimes^{\mathbb L }  \mathcal O_{Z(b_2)^h} ) \ =& \ \length \left( R / (r_1 T - s_1 ) \otimes_R R / (r_2T - s_2) \right)_x \\  
 =& \ \length W / (\epsilon \mu^{-1} (r_1' s_1 - r_1 s_1' ) ) \\
 =&\  \ord_p (\epsilon)
 \end{align*}
 as required. 
 
When $\Lambda$ is a type 2 vertex lattice, we may fix a basis $\{ e, f \}$ such that $h \sim p^{-1}( \begin{smallmatrix} & \delta \\ -\delta & \end{smallmatrix})$. Writing 
\[\beta_i \ = \ r_i \,e \ + \ s_i \, f \]
as before, Proposition 3.11 of \cite{San1} tells us that the local equation for $Z(b_i)^{h} $ is given by
\[ s_i' \, T\ +  \ r_i' \ = \ 0. \]
for $i=1,2$. The result follows from similar considerations to the previous case. 
\end{proof}
\end{proposition}

We have now completed the computation of the pairing $\la Z(b_1), Z(b_2) \ra $. The next step is to show that this pairing depends only the matrix of inner products, as in the following lemma. 

\begin{lemma} \label{TdiagLemma}
Suppose $ \mathbf b = [b_1, b_2] \in \mathbb V$ is a linearly independent pair of vectors, and let 
\[ T\ = \ (\mathbf b, \mathbf b) \ = \ \begin{pmatrix} (b_1, b_1) & (b_1, b_2) \\ (b_2, b_1) & (b_2, b_2) \end{pmatrix}\] 
denote the matrix of inner products. Set $m_i = \ord_p (b_i, b_i)$, and let $\Lambda_1, \Lambda_2$ denote the central lattices of $b_1$ and $b_2$ respectively with $d := d(\Lambda_1, \Lambda_2)$. Replacing $T$ by a $GL_2(o_{k,p})$-conjugate if necessary, we may further assume $m_1 \leq m_2$. 
Then
\begin{compactenum}[(i)]
\item $T \in \Herm_2(o_{k,p})$ if and only if $ \mathcal{B}(b_1) \cap \mathcal{B}(b_2) \neq \emptyset$. 
\item If $\mathcal{B}(b_1) \subset \mathcal{B}(b_2)$ and $\Lambda_1\neq \Lambda_2$, then $T$ is $GL_2(o_{k,p})$-conjugate to the matrix $\diag(p^a, p^b)$ 
with 
\[ a  \ = \ m_2  \ - \ d, \qquad b \ = \ m_1. \]
\item If $\mathcal{B}(b_1) \not\subset \mathcal{B}(b_2)$ and $\mathcal B(b_1) \cap \mathcal B(b_2) \neq \emptyset$,  then $T \sim {\diag}(p^r, p^r)$ where
\[ r =  \frac{m_1+ m_2 - d}2. \] 
\item If $\Lambda_1 = \Lambda_2$, then $T \sim {\diag}(p^a, p^b) $ where 
\[ a  = m_2 + 2\ \ord_p (\beta_2, \beta'_1), \qquad b = m_1 ; \]
here $\beta_2$ and $\beta'_1$ are as in \refProp{hhProp}. 
\end{compactenum}
\begin{proof} If $m_1 < 0$, then clearly $T \notin \Herm_2(o_{k,p})$ and by definition  $\mathcal B(b_1) = \emptyset$, so the lemma holds trivially in this case. Hence we assume that $0 \leq m_1 \leq m_2$. 

\vspace{1em}
\noindent\fbox{Case 1: $\Lambda_1 \neq \Lambda_2$}
\vspace{1em}

 First suppose that  $\Lambda_1$ is self-dual, and so $m_1$ is even. We can assume that $\Lambda_1$ and $\Lambda_2$ lie on the `standard lattice chain'. In other words, there exists a basis $\{ e_1, f_1 \}$ for $\Lambda_1$ such that
\begin{compactenum}[(a)]
\item with respect to this basis,  $ h \sim (\begin{smallmatrix} & \delta \\ -\delta & \end{smallmatrix})$, where $\delta \in o_{k,p}^{\times}$ with $\delta' = -\delta$, and
\item  the lattice $\Lambda_2$ has basis $\left\{ e_2, f_2 \right\}$ where
\[ (e_2, f_2) \ = \ \begin{cases} (p^{-k} e_1, \ p^k f_1), & \text{if } d = 2k \text{ is even} \\ (p^{-k-1} e_1, \ p^{k} f_1), & \text{if } d = 2k + 1\text{ is odd.} \end{cases} \]
\end{compactenum}
Set $ t_1 = \frac{m_1}{2}$ and $ t_2 = \left\lfloor \frac{m_2 + 1}{2} \right\rfloor $. Then  $T$ has the form
\begin{equation} \label{TFormEqn}
 T \ = \ \begin{pmatrix} p^{m_1} \cdot (\mathrm{unit}) & p^{s} \cdot (\mathrm{unit}) \\  p^{s} \cdot (\mathrm{unit}) & p^{m_2} \cdot (\mathrm{unit}) \end{pmatrix} 
\end{equation}
where 
\[ s := \begin{cases} t_1 + t_2 - k & \text{if } d = 2k > 0 \text{ is even} \\ t_1 + t_2 - k  - 1 & \text{if } d = 2k+1 \text{ is odd}. \end{cases} \]
Since $m_1$ is even, we have $m_2 \equiv d \pmod 2$ and so in fact
\[ s = \frac{m_1 + m_2 - d}{2} \] 
in both cases. Hence
\[ T \in \mathrm{ Herm}_2(o_k) \ \iff \ d \leq m_1 + m_2 \ \iff \ \mathcal B (b_1) \cap \mathcal B(b_2) \neq \emptyset , \]
which proves (i), at least when $\Lambda_1 \neq \Lambda_2$ (on the other hand if $\Lambda_1 = \Lambda_2$, then (i) is trivial). 

Note that in general if $T \sim {\diag}(p^a, p^b)$ with $a \geq b \geq 0$, the numbers $a$ and $b$ are characterized by the fact that 
\begin{compactitem}
\item $b$ is the lowest valuation appearing among the entries of $T$; and
\item $a+b  = \ord_p \det T$.
\end{compactitem}
We also observe that 
\[ \mathcal B( b_1 ) \subset \mathcal B(b_2) \iff m_1 \leq m_2 - d \iff m_1 \leq s. \]
Thus  statements \emph{(ii)} and \emph{(iii)} of the lemma follow from a moment's contemplation of   \eqref{TFormEqn}, and begging the reader's forbearance for yet another instance of this refrain, the proof when $\Lambda_1$ is a type 2 lattice is entirely analogous.

\vspace{1em}
\noindent\fbox{Case 2: $\Lambda_1 = \Lambda_2 $}
\vspace{1em}

Abbreviate $\Lambda_1 = \Lambda_2 = \Lambda$. Recall that we defined $t_i = \lfloor \frac{m_i + 1}{2} \rfloor$ and $\beta_i = p^{-t_i} b_i$ as in \refProp{hhProp}, so that $\beta_i \in \Lambda - p \Lambda$ and $\ord_p (\beta_i, \beta_i)$ is equal to $0$ or $-1$, depending on whether or not $m_1$ and $m_2$ are both even or both odd. 

We may choose an element $\beta_1' \in \Lambda - p \Lambda$ such that $(\beta_1, \beta_1') = 0$ and $(\beta_1', \beta_1') =  (\beta_1, \beta_1)$ with $\{ \beta_1, \beta_1'\}$ forming a basis for $\Lambda$. Write
\[ \beta_2 = \mu \cdot \beta_1 + \epsilon \cdot \beta_1' \]
where $\mu, \epsilon \in o_{k,p}$ with at least one of them being a unit. 
Then 
\[ T \ = \ (\beta_1, \beta_1) \cdot \begin{pmatrix} p^{2t_1} & p^{t_1 + t_2} \mu' \\ p^{t_1 + t_2} \mu & p^{2t_2}( n(\mu) + n(\epsilon) )  \end{pmatrix} \]
The smallest valuation appearing is 
\[ 2 t_1 \ + \  \ord_p (\beta_1, \beta_1) \ = \ m_1, \]
and the determinant of $T$ has valuation
\[ \ord_p \det (T) \ = \ 2 \ord_p (\beta_1, \beta_1)  + 2 (t_1 + t_2) + \ord_p n(\epsilon) \ = \ m_1 + m_2 + 2 \ \ord_p (\beta_2, \beta_1') .\]
This proves \emph{(iv)}. 
\end{proof}

\end{lemma}

\begin{corollary} \label{mainLocGeomCor}
Suppose $\mathbf b = [b_1, b_2] \in \mathbb V^2$ is a linearly independent pair of vectors with $T = (\mathbf b, \mathbf b)$  as above. Then the pairing $\deg Z(\mathbf b) := \la Z(b_1), Z(b_2) \ra$ depends only on the $GL_2(o_{k,p})$-conjugacy class of $T$. More precisely, if $T \in \Herm_2(o_k)$ and $T \sim {\diag}(p^a, p^b)$ with $a \geq b$, then 
\begin{equation} \label{mainLocGeomCorEqn}
 \deg Z(\mathbf b) \ = \ \la Z(b_1) , \, Z(b_2) \ra \ = \ \frac{a+b}{2} \ - \ p \left( \frac{p^b -1}{p-1} \right) \ =: \ \mu_p(T). 
\end{equation}
\begin{proof}
This follows immediately from \refThm{localDegThm} and \refLemma{TdiagLemma}.
\end{proof}
\end{corollary}
\section{A closed-form formula for certain representation densities}
Suppose $S \in \Herm_m(o_{k,p})$ and $T \in \Herm_n(o_{k,p})$ for some integers $m$ and $n$, where we continue to work with the localization $o_{k,p}$ at an unramified prime $p$. Then we may consider the \emph{representation density}
\begin{equation} \label{repDenDefEqn}
 \alpha(S, T) \ := \ \lim_{k \to \infty} p^{-kn(2m-n)} \ \# \{ x \in M_{m,n}(o_{k,p} / p^k)  \ | \ {}^t(x') S x \equiv T \pmod{p^k} \}, 
\end{equation}
which will play a crucial role in our determination of the Fourier coefficients of Eisenstein series. 
As this quantity depends only the $GL_n(o_{k,p})$ (resp. $GL_m(o_{k,p})$)-conjugacy classes of $T$ and $S$ respectively, we may suppose that they are diagonal.
Moreover, for a fixed $S \in \Herm_m(o_{k,p})$, we write
\[ S_r \ := \ \begin{pmatrix} S & \\ & \Id_r \end{pmatrix} \ \in \ \Herm_{m+r}(o_{k,p}). \]
As we will see shortly, there is a polynomial $F(S, T; X) \in \Q[X]$ such that
\[ \alpha(S_r, T) \ = \ F(S,T; (-p)^{-r}) . \]
The aim of this section is to prove the following closed-form expression for $F(S,T,X)$ when $T \in \Herm_2(o_{k,p})$ such that $\ord_p \det(T)$ is even,  and $S = \diag(p,1)$. This can be seen as the counterpart to Nagaoka's  result \cite{Nag}, which considers the case $S = \Id_2$. 

\begin{proposition} \label{RepDensProp}
Set $S = \diag(p,1)$. Then
\begin{compactenum}[(i)]
\item 
\[ F(S, \Id_2; X) = \frac{(1-X)(X+p)}{p}  \ \ \text{and} \ \  F(S,( \sm{ p & \\ & p });X) =  \frac{(1-X)(X+p)}{p} \left( X^2 - (p^2 - p)X + 1 \right)  \]
\item Let $T = \diag(p^a, p^b)$, where $a \geq b \geq 0$ and $a+b$ is even. 
 Set 
\[ \epsilon = \begin{cases} 0, & \text{if } b \text{ is even} \\ 1, & \text{if } b \text{ is odd. } \end{cases} \]
and define
\[ F_{\epsilon}(X) \ := \ F(S, \ p^{\epsilon} \cdot  \Id_2; \ X) \ = \ \frac{(1-X)(X+p)}{p} \left( X^2 - (p^2 - p)X + 1 \right)^{\epsilon}. \]
Then
\begin{align*}
 F(S,T;X) \ =&  \ F_{\epsilon}(X) +  \frac{(X-1)(X+p)}{X-p} \Bigg\{ (pX)(1-p) \frac{ (pX)^b - (pX)^{\epsilon} }{pX - 1}  + X^2 (p - p^{-1}X)  \frac{X^{2b} - X^{2 \epsilon}}{X^2-1} \\
 & \ + \left( -p^{b+1}(X-1) + p X^{b+1} - p^{-1} X^{b+2} \right) \cdot \frac{X^{a+1} - X^{b+1}}{X^2-1}  \Bigg\}.
 \end{align*}
 \end{compactenum}

 \end{proposition}

The proof of this proposition, which appears at the end of this subsection, amounts to specializing Hironaka's explicit formula \cite[Theorem II]{Hiro} to the case at hand. We shall briefly review Hironaka's notation in the general case. First, given integers $u \geq v \geq 0$, we define a symbol
\[ \begin{bmatrix} u \\ v \end{bmatrix} \ := \ \frac{ \prod_{i=1}^u \left( 1 - (-p)^{-i} \right) }{  \prod_{i=1}^v \left( 1 - (-p)^{-i} \right) \ \prod_{i=1}^{u-v} \left( 1 - (-p)^{-i} \right) }\] 
Let 
\[ \Z^k_{\succ0} \ := \ \left\{ a = (a_1, \dots a_k) \in \Z^k  \ | \ a_1 \geq a_2 \geq \dots a_k \geq 0  \right\} \]
denote the set of non-increasing non-negative vectors in $\Z^k$. Given $a = (a_1, \dots, a_k) \in \Z^k_{\succeq0}$,  set
\[ \tilde a \ := \ ( a_1 + 1, \dots,  a_k + 1 )  \in \Z^k_{\succeq0} \]
and for any $i \geq 1$, we let 
\[ a_i' := \# \left\{ j \ | \ a_j  \geq i \right\}. \]
Next, suppose $\lambda , \mu \in \Z^k_{\succ0}$. For an integer $j \geq 1$, we define\footnote{There is a typographical error in the statement of Theorem II of \cite{Hiro}: the corresponding formula in loc.\ cit.\ appears {without} a necessary tilde in the exponent of $(-p)$.}
\[ I_j ( \mu, \lambda) \ := \ \sum_{i = \mu'_{j+1} }^{\min \left( (\tilde \lambda)'_{j+1},  \ \mu'_j\right)} \ (-p)^{i \left( 2 (\tilde\lambda)'_{j+1}   + 1  - i \right)/2}  
 \cdot \begin{bmatrix} (\tilde \lambda)'_{j+1} - \mu'_{j+1} \\[8pt]  (\tilde \lambda)'_{j+1} - i \end{bmatrix} \cdot \begin{bmatrix} (\tilde \lambda)'_{j} - i \\[8pt]  (\tilde \lambda)'_{j} - \mu'_{j}  \end{bmatrix}.  \]
We also define a partial order on $\Z^k_{\succ0}$ by declaring 
\[ a \leq b \ \iff \ a_i \leq b_i \ \text{for all }  i = 1, \dots, k. \]
Finally, we put
\[ |a| := \sum_{i=1}^k a_i, \qquad  \text{and} \qquad n(a) := \sum_{i=1}^k (i-1) a_i. \]
With all of this notation in place, we can state Hironaka's formula:

\begin{theorem}[{\cite[Theorem II]{Hiro}}]  \label{HiroThm}
Let $\lambda \in \Z^n_{\succ0} $ and $\xi \in \Z^m_{\succ0}$ with $m \geq n$. Suppose that $T_{\lambda} \in \Herm_n(o_k)$ is $GL_n(o_k)$-equivalent to ${\diag}(p^{\lambda_1}, \dots p^{\lambda_n})$, and that $S_{\xi}$ is equivalent to ${\diag}(p^{\xi_1}, \dots , p^{ \xi_m})$. Then 
\begin{align*}
\alpha(S_{\xi},T_{\lambda}) \ = \ \sum_{ \substack{\mu \in \Z^n_{\succ0} \\  \mu \leq \tilde\lambda}} (-1)^{|\mu|} (-p)^{-n(\mu) + (n-m-1)|\mu| + \la \xi', \mu' \ra}  \cdot \prod_{j \geq 1} I_j(\mu, \lambda),
\end{align*}
where $\la \xi' , \mu' \ra = \sum_{i \geq 1} \xi'_i \ \mu'_i$.  \qed
\end{theorem}

We now specialize this formula to our case of interest:
take $n = 2$, $m = r+2$, $\lambda = (a, b)$ with $a+b$ even, and  $\xi = (1, 0, \dots 0) \in \Z^{r+2}_{\succ0}$.  If we put
\[ T  = T_{\lambda} := \ \diag(p^{a}, p^{b}), \qquad \text{and} \qquad  S = \diag(p,1), \]
then in particular $S_r = S \oplus 1_r = S_{\xi}$. Taking $X := (-p)^{-r}$ in Hironaka's theorem gives us the expression 
\begin{align} \label{FPolyEqn}
 F(S, T ; X) \ = \ \sum_{c = 0}^{ a +1} \ \sum_{d = 0}^{\min(c,b+1)} (-1)^d p^{-2d - c} X^{c+d} (-p)^{\epsilon_c + \epsilon_d} \cdot \prod_{j \geq 1} I_j \left( (\sm{c \\ d}), \lambda \right), 
\end{align}
where $\epsilon_c$ is equal to $0$ if $c=0$, and is equal to $1$ if $c \geq 1$; we define $\epsilon_d$ likewise. 

Our first step towards giving a closed form expression for \eqref{FPolyEqn} is the following table of values for $I_j(\cdot, \cdot)$, which is easily proven by explicit computation. 
 \begin{lemma} \label{IjLemma}
 Suppose $\ell = (\alpha, \beta)$, and $(c,d) \leq (\alpha + 1, \beta + 1)$ are integers with $c \geq d$. 

(i) If  $c > \beta+1 \geq d$, then
\[ I_j \left( \binom{c}{d}, \ell \right) = 
\begin{cases}
-p^3, & 1 \leq j < d \\
p^2 -p^3, & j=d, \ d < \beta+1 \\
p^2, & d+1 \leq j < \beta+1\\
-p, & \beta+1 \leq j < c \ (\text{including } d = \beta+1) \\
1-p, & j = c < \alpha + 1 \\
1, & j = c= \alpha + 1  \text{ or } j > c.
\end{cases}
\]

(ii) If $\beta \geq c > d$, then
\[ I_j \left( \binom{c}{d}, \ell \right) = 
\begin{cases}
-p^3, & 1 \leq j < d \\
p^2 -p^3, & j=d \\
p^2, & d+1 \leq j < c\\
(1+p^2)(1 - p^{-1}) , & j = c \\
1, & j> c.
\end{cases}
\]

(iii) If $c = d \leq \beta$, then
\[ 
 I_j \left( \binom{d}{d}, \ell \right) = 
\begin{cases}
-p^3, & 1 \leq j < d \\
(1+p^2)(1 - p) , & j = d \\
1, & j > d.
\end{cases}
\]

(iv) If $c = \beta + 1$, $d \leq \beta$, then
\[ 
 I_j \left( \binom{\beta + 1}{d}, \ell \right) = 
\begin{cases}
-p^3, & 1 \leq j < d \\
p^2 -p^3, & j=d \\
p^2, & d+1 \leq j < \beta + 1\\
-p^{-1} + 1 - p, & j = \beta + 1 < \alpha + 1 \\
1 - p^{-1}, & j = \beta+ 1 = \alpha + 1 \\
1, & j > \beta + 1.
\end{cases}
\]

(v) If $ c = d = \beta + 1$, then
\[ 
 I_j \left( \binom{\beta + 1}{\beta+1}, \ell \right) = 
\begin{cases}
-p^3, & 1 \leq j < \beta + 1 \\
1 - p, & j = \beta + 1 < \alpha + 1\\
1, & j > \beta + 1 \text{ or } j = \beta + 1 = \alpha + 1
\end{cases}
\]
\qed
 \end{lemma}
 
 Next, we give a pair of lemmas describing inductive formulas for the representation densities:
\begin{lemma}  \label{IndStepA}
Suppose $T^+ = \diag(p^{a+2}, p^b)$ and $ T  = \diag(p^a, p^b)$  for a pair of integers $a,b$ such that $a+b$ is even. 
Then 
\[ F(S, T^+; X) - F(S, T; X) = \frac{X^{a+1}(X+p)(X-1)}{X-p} \left( -p^{b+1} (X-1) + p X^{b+1} - p^{-1} X^{b+2} \right), \]
where $S = \diag(p,1)$. 
\begin{proof}
Let $\Lambda = (a+2, b)$ and $\lambda = (a,b)$, and abbreviate
\[ F(\Lambda) \ := \ F( S, T^+, X) \qquad \text{and} \qquad F(\lambda) := F( S,T, X).\]
Note that if $c \leq a$, then $ I_j \left( (\sm{c \\ d}), \Lambda \right) = I_j \left((\sm{c \\ d}), \lambda\right)$, and so 
\begin{align*}
F(\Lambda) - F(\lambda) &=  - p^{-a} X^{a+1} \sum_{d =0}^{b+1} (-1)^d p^{-2d} X^{d} (-p)^{\epsilon_d}  \cdot \left\{\prod_{j}  I_j \left( (\sm{a+1 \\ d}) , \Lambda \right) - \prod_j  I_j \left( (\sm{a+1 \\ d}), \lambda \right) \right\} \\
& \qquad -p^{-(a+1)} X^{a+2}  \sum_{d =0}^{b+1} (-1)^d p^{-2d} X^{d} (-p)^{\epsilon_d} \prod_j I_j \left( (\sm{a+2 \\ d}), \Lambda \right) \\
& \qquad \qquad -p^{-(a+2)} X^{a+3}  \sum_{d =0}^{b+1} (-1)^d p^{-2d} X^{d} (-p)^{\epsilon_d} \prod_j I_j \left(  (\sm{a+3 \\ d}), \Lambda \right) \\
&= - p^{-a}X^{a+1} \sum_{d=0}^{b+1} (-1)^d (-p)^{-2d + \epsilon_d} X^{d}   \\
& \quad \times \Bigg\{  \prod_j I_j \left( (\sm{a+1 \\ d}) ,\Lambda \right) -  \prod_j I_j \left( (\sm{a+1 \\ d}) ,\lambda \right) + \frac{X}{p} \prod_j I_j \left( (\sm{a+2 \\ d}) ,\Lambda \right)  + \frac{X^2}{p^2}  \prod_j I_j \left( (\sm{a+3 \\ d}) ,\Lambda \right) \Bigg\}.
\end{align*}
The term in curly braces can be computed explicity using \refLemma{IjLemma}, and the result readily follows.
\end{proof}

%
%
%
%
%
\end{lemma}

\begin{lemma} \label{IndStepB}
Suppose $T^+ = \diag(p^{b+2}, p^{b+2})$ and $T = \diag(p^{b+2}, p^{b})$, and set $S = \diag(p,1)$. Then
\begin{align*}
 F(S,T^+; X) - F (S,T; X) \  = \  &\frac{X^{b+1} (X+p)(X-1)}{X-p}  \\
 & \qquad \times  \left[ \left( (1+p-p^2)X - p \right) p^{b+1} +  \frac{p^2 - X}{p} X^{b+3} \right] .
 \end{align*}

\begin{proof}
Set $\Lambda := (b+2, b+2)$ and $\lambda = (b+2, b)$, and abbreviate
\[ F(\Lambda) := F(S, T^+;X), \qquad  F(\lambda) := F(S,T;X). \]
Note that for $c \leq b$ and any $j$, we have  $ I_j \left( (\sm{ c \\ d}), \Lambda \right)  = I_j \left( (\sm{c \\ d}), \lambda \right),$
and so
\begin{align*}
F(\Lambda) - F(\lambda) = \sum_{d=0}^{b+1} & (-1)^d  (-p)^{-2d + 1 + e_d} X^d \left\{ \sum_{c=b+1}^{b+3} p^{-c} X^c \left( \prod_{j=1}^c I_j \left( (\sm{c \\ d}) , \Lambda \right) - \prod_{j=1}^c I_j \left(  (\sm{c \\ d}), \lambda \right) \right) \right\} \\
 & + (-1)^{b+2}(-p)^{-2b-2} X^{b+2} \left\{ \sum_{c= b+2}^{b+3} p^{-c}  X^c \prod_{j=1}^c I_j \left( (\sm{c \\ b+2}), \Lambda \right) \right\} \\
 & \quad + (-1)^{b+3}(-p)^{-2b-4} X^{b+3} \left\{ p^{-(b+3)} X^{b+3} \prod_{j=1}^{b+3} I_j \left((\sm{b+3\\  b+3}), \Lambda \right) \right\},
\end{align*}
where 
\[ e_d := \begin{cases} 0, & \text{if } d = 0 \\ 1, & \text{if } d > 0. \end{cases} \]
The terms in curly braces can again be computed via \hyperref[IjLemma]{Lemma \ref{IjLemma}}, and the proposition follows after straightforward algebraic manipulations.
\end{proof}
\end{lemma}

\begin{proof}[Proof of  \refProp{RepDensProp}]
First, we note that the formulas for $F(S, \Id_2; X)$ and $F(S, p \Id_2; X)$, which correspond to the cases $(a, b) = (0,0)$ and $(a, b) = (1,1)$ respectively, can be verified directly via \eqref{FPolyEqn}, proving \emph{(i)}. 

Next,  for notational convenience set
\[ F(r, s) := F\left(S,  (\begin{smallmatrix}p^{r} & \\  & p^{s} \end{smallmatrix}); X \right), \]
where as usual, $r \geq s \geq 0$, and $r + s$ is even.
By applying  \refLemma{IndStepA} and 
\refLemma{IndStepB}
in sequence, we have that for any $r \geq 0$, 
\begin{align*}
F( r+2, r+2)  \ - \ F(r, r) \ = \ & \frac{X^{r+1} (X+p)(X-1)}{X-p} \\
& \qquad \times \Big\{ p^{r+1}(1-p)( pX + 1) \ +\ X^{r+1}(1+X^2)(p-p^{-1}X)  \Big\},
\end{align*}
which upon repeated application yields the formula
\begin{align*}
 F(b, b) \  =& \  F(\epsilon, \epsilon) + \sum_{i=1}^{(b-\epsilon) / 2} F(2i+\epsilon, 2i+\epsilon) - F(2i-2+\epsilon, 2i-2 + \epsilon) \\
=&  \ F({\epsilon}, {\epsilon})  \  + \  \frac{(X+p)(X-1)}{X-p} \Bigg[ (pX)(1-p) \frac{ (pX)^b - (pX)^{\epsilon} }{pX - 1}   \\
 & \qquad \qquad \qquad\qquad\qquad \qquad \qquad+ X^2 (p - p^{-1}X)  \frac{X^{2b} - X^{2 \epsilon}}{X^2-1} \Bigg].
 \end{align*}
On the other hand, for $a \geq b \geq 0$ with $a+b$ even, the repeated application of \refLemma{IndStepA} yields the relation
\[ F(a,b) = F(b,b) + \frac{(X+p)(X-1)}{X-p} \Bigg[ -p^{b+1}(X-1) + p X^{b+1} - p^{-1} X^{b+2} \Bigg] \cdot \frac{X^{a+1} - X^{b+1}}{X^2-1},\]
which then implies the proposition after a little straightforward algebra.
\end{proof}

%
The motivation for our calculations so far is to facilitate computing the derivative
\[ \alpha'(S,T) \ := \ - \left[ \frac{\partial}{\partial X} F(S, T; X)\right]_{X = 1} \]
\begin{corollary}  \label{repDenFinalCor}
Let $T \in \Herm_2(o_{k,p})$ such that $T \sim \diag(p^a, p^b)$, where $a \geq b \geq 0$ and $a+b$ is even. Then
\[  \frac{p}{(p+1)^2}  \Big[  \alpha' \big(  (\begin{smallmatrix} p & \\ & 1\end{smallmatrix}) ,  T \big) \ + \frac{p^2}{1-p^2} \alpha(Id_2, T ) \Big] = \frac{a+b}{2} - p \left( \frac{p^b - 1}{p-1} \right) \ =: \ \mu_p(T) . \]
\begin{proof}
We recall Nagaoka's formula \cite{Nag} for the representation density for $S = \Id_2$:
\[ F(\Id_2, T; X) \ = \ (1 + p^{-1} X) (1 - p^{-2} X) \sum_{\ell = 0}^b (pX)^{\ell} \sum_{k = 0}^{a + b - 2 \ell} (-X)^k .\]
Thus
\[ \alpha(\Id_2, T) \ = \ F(\Id_2, T ; X)_{X = 1} \ = \ (1+p^{-1})(1-p^{-2}) \frac{ p^{b+1} - 1 }{p-1} .\]
On the other hand, the derivative $\alpha'( ( \sm{ 1 & \\ & p} ) , T) $ can be computed directly from \refProp{RepDensProp}, and the proposition follows  via straightforward algebraic manipulation. 
\end{proof}
\end{corollary}

\section{Global cycles and Eisenstein series}

In this section, we turn to global aspects: we describe the Shimura varieties of interest and their global cycles, the construction of the relevant Eisenstein series, and prove our main theorem relating the two. Our presentation and approach is closely modelled on the account given by Kudla and Rapoport in \cite{KRunnglob}: we shall refer freely to the results therein and content ourselves in the present work to describing the necessary modifications to their arguments as the need arises.

\subsection{Preliminaries on Hermitian spaces} {\ \\}

Here we recall some basic notions about Hermitian spaces. Let $V$ be a Hermitian space over $k$ of signature $(r,s)$. To every rational place $\ell \leq \infty$, there is an associated  \emph{invariant}
\[ \inv_{\ell}(V) \ := \ \left( \det(V), \ \Delta \right)_{\ell}  \ = \ \begin{cases} 1, & \text{if }  \det(V) \in N(k_{\ell}^{\times}) \\ -1, & \text{if }  \det(V) \notin N(k_{\ell}^{\times}). \end{cases}\]
where $\det(V) = \det( (\mathbf v, \mathbf v) ) \in \Q_{\ell}^{\times}  / N(k_{\ell}^{\times})$ is the determinant of the matrix of inner products of any basis $\mathbf v = \{ v_1, \dots, v_{r+s} \}$ of $V$.
In particular, $\inv_{\infty}(V) = (-1)^{ s}$ and if $\ell$ is split, then $\inv_{\ell}(V) = 1$. These invariants satisfy the product formula
\[ 1 \ = \ \prod_{\ell \leq \infty} \ \inv_{\ell}(V). \]
Using the same definition, we may also define the  local invariant $\inv_{\ell}(V_{\ell}) = (\det V_{\ell}, \Delta)_{\ell}$ Hermitian space $V_{\ell}$ over $k_{\ell}$. When $\ell$ is a finite prime, two local Hermitian vector spaces are isometric if and only if they have the same dimension and their invariants are equal. If ${\ell} = \R$, there is a unique isometry class for each signature. 

Suppose that we are given a collection of signs $(a_p)_{p \leq \infty}$, almost all of which are equal to 1, and satisfying the product formula $\prod_{p \leq \infty} a_p = 1$. Then for any pair of integers $(r, s)$ such that $ (-1)^s  = a_{\infty}$,  there exists a hermitian space $V$ over $k$ of signature $(r,s)$ such that $\inv_{\ell}(V) = a_{\ell}$ for all $\ell$, and furthermore $V$ is unique up to isometry; in other words, the Hasse principle holds for Hermitian vector spaces. 

Finally, let $L$ be an $o_k$-Hermitian lattice, i.e. a projective $o_k$-module of finite rank equipped with an $o_k$-Hermitian form, and set $V = L \otimes_{\Z} \Q$. The \emph{genus} $[L]$ is the set of isomorphism classes of lattices $M \subset V$ such that 
\[ M \otimes_{\Z} \widehat{\Z} \ \simeq \ L \otimes_{\Z} \widehat{\Z} \]
as Hermitian $\widehat{o_k}$-modules. In other words, $M \in [L]$ if and only if $M_{\Q} = V$ and for every finite prime $\ell$, there exists an element $g_{\ell} \in U(V_{\ell})$ such that $g_{\ell}(M_{\ell}) = L_{\ell}$. 

Suppose $\ell \neq 2$ and $\Lambda$ is a self-dual $o_{k,\ell}$-lattice. Then by \cite{Jac}, the set of self-dual lattices in $V_{\ell} = \Lambda_{\Q_{\ell}}$ forms a single $U(V_{\ell})$-orbit. Furthermore if $\ell$ is inert, then the existence of a self-dual lattice forces $\inv_{\ell}(V_{\ell}) = 1$. 

When $\ell\neq 2$ is an inert prime, we say an $o_{k, \ell}$ lattice $\Lambda$ is \emph{almost self-dual} if $\Lambda^{\#} / \Lambda \simeq \F_{\ell^2}$, where $\Lambda^{\#}$ is the dual lattice.  In this case, $\inv(\Lambda_{\Q_{\ell}} ) = -1$, and the set of almost self-dual lattices in $\Lambda_{\Q_{\ell}}$ again forms a single orbit under the action of $U(\Lambda_{\Q_{\ell}})$. 
\subsection{Global moduli problems and $p$-adic uniformizations} \label{GlobGeomSec}{ \ \\}

Fix a squarefree integer $d$ whose prime factors are all inert in $k$. We define the moduli space of almost-principally-polarized abelian surfaces as follows:

\begin{definition}  \label{Md(1,1)Def}
 Let $\mathrm M^d_{(1,1)}$ denote the Deligne-Mumford stack over $\Spec(o_k)$ defined by the following moduli problem: for a scheme $S$ over $\Spec(o_k)$, the points $\mathrm M^d_{(1,1)}(S)$ parametrizes the category of tuples $\underline A = (A, i_A, \lambda_A)$, where
\begin{compactenum}[(i)]
\item $A$ is an abelian surface over $S$;
\item $i_A \colon o_k \to \End(A)$ is an $o_k$-action satisfying the following \emph{signature $(1,1)$ condition}:  on  (the locally free $\mathcal O_S$-module) $\Lie(A)$, the induced action has characteristic polynomial
\begin{equation} \label{sig11Eqn}
 \det( T \ - i_A(a)|_{\Lie(A)} ) \ = \ (T - a) (T - a')  \ \in \ \mathcal O_S [ T] \qquad \text{for all } a \in o_k;
\end{equation}
\item and finally, $\lambda_A$ is a polarization such that the induced Rosati involution ${}^*$ satisfies 
\[ i_A(a)^* = i_A(a'). \]
 In addition, we require that
\[  \ker(\lambda_A) \subset A[d], \qquad \text{and} \qquad |\ker(\lambda_A)| = d^2.    \]
\end{compactenum}
\end{definition}

\begin{proposition} $\mathrm M^d_{(1,1)}$ is flat over $\Spec(o_k)$ and smooth over $\Spec o_k[ (d \cdot \Delta )^{-1}]$.
\begin{proof}  This follows from combining the results of \cite[\S 2]{KRunnglob}, for primes away from $d$,  with \cite{KRDrinfeld} for those primes dividing $d$.
\end{proof}
\end{proposition}

We also  set $\mathcal E$ to be the DM stack over $\Spec(o_k)$ that parametrizes principally polarized elliptic curves with multiplication by $o_k$. More precisely, for a scheme $S / o_k$, the points $\mathcal E(S)$ parametrize tuples $\underline E = (E, i_E, \lambda_E)$, where 
\begin{compactenum}[(i)]
	\item $E/S$ is an elliptic curve;
	\item  $i_E\colon o_k \to \End(E)$ satisfies the \emph{signature (1,0) condition}: concretely, this means that  on $\Lie(E)$, 
		\[ i_E(a)|_{\Lie(E)} \ = \ \tau_S (a) \qquad \text{for all } a \in o_k\]
	where $\tau_S \colon o_k \to \mathcal O_S$ is the structural morphism; and
 	\item $\lambda_E$ is a principal polarization whose corresponding Rosati involution induces Galois conjugation on $i_E(o_k)$. 
\end{compactenum}
	This is the stack denoted by $\mathcal M(1,0) ^{naive}$ in the notation of \cite{KRunnglob}; in particular, it is proper over $\Spec(o_k)$ of relative dimension 0.

Finally, we set
\[ \mathcal M \ := \ \mathcal E \ \times_{\Spec o_k} \ \mathrm M^d_{(1,1)}. \] 

Next, we describe the \emph{Kudla-Rapoport} cycles on $\mathcal M$, as introduced in \cite[\S 2]{KRunnglob}.
Given a scheme $S / o_k$ and a point $(\underline E, \underline A) = (E, i_E, \lambda_E, A, i_A, \lambda_A) \in \M(S)$,  the space of \emph{special homomorphisms}
\[  \Hom_{o_k, S}(E, A) \]
can be equipped with a positive-definite hermitian form, defined by the formula
\begin{equation} \label{SpecCycleHermFormEqn}
 (x, y) \ := \ \lambda_E^{-1} \circ y^{\vee} \circ \lambda_A \circ x \ \in \ \End_{o_k}(E) \ \simeq \ o_k.
\end{equation}

\begin{definition} (i) Let $m \in \Z_{>0}$. We define the special cycle $\Zed(m)$ to be the moduli space over $\Spec o_k$ whose $S$ points parametrize triples
\[ \Zed(m) (S) \ = \ \left\{ (\underline E, \, \underline A ; \, y ) \ | \ (\underline E, \, \underline A) \in \M(S) \ \text{ and } y \in\Hom_{o_k, S}(E, A) \text{ with } (y,y) = m \right\} \]

(ii) Suppose $T \in \Herm_2(o_k)$. We define $\Zed(T)$ to be the moduli space over $\Spec o_k$ whose $S$ points parametrize tuples
\[ \Zed(T) (S) \ = \ \left\{ (\underline E, \,  \underline A ; \, \mathbf y ) \ | \ (\underline E,  \, \underline A) \in \M(S) \ \text{ and } \  \mathbf y \in \Hom_{o_k, S}(E, A)^2 \text{ with }  (\mathbf y, \, \mathbf y) =  T \right\}, \]
where, for $ \mathbf y = (y_1, y_2)$, the matrix $(\mathbf y, \mathbf y)$ is the matrix of inner products
\[ (\mathbf y, \mathbf y) \ := \ \begin{pmatrix} (y_1, y_1) & (y_1, y_2) \\  (y_2, y_1) &  (y_2, y_2)  \end{pmatrix} \ \in \ \Herm_2(o_k) .\]

\end{definition}
Both moduli problems are represented by DM stacks. Furthermore the natural forgetful maps to $\M$ are finite and unramified, cf. \cite[Proposition 2.9]{KRunnglob}, and so their images can be viewed as cycles on $\M$. In the sequel, we shall abuse notation and use the symbols $\Zed(m)$ and $\Zed(T)$ to denote both the representing stacks and the corresponding cycles on $\M$, and hope that the intended meaning can be inferred from the context.

The aim of this section is to compute the \emph{arithmetic degree} of a cycle $\Zed(T)$, which we define as follows: suppose 
\[ T \ = \ \begin{pmatrix} m_1 & a \\ a' & m_2 \end{pmatrix} \  \in  \ \mathrm{Herm}_2(o_k), \]
where $m_1, m_2 \in \Z_{>0}$. A glance at the definitions above reveals that 
\[ \Zed(m_1) \ \times_{\M} \ \Zed(m_2) \  = \ \coprod_{T'  = (\begin{smallmatrix} m_1 & * \\ * & m_2 \end{smallmatrix} ) } \ \Zed(T') \ \supset \Zed(T).\]
Suppose $T$ is positive definite. As we shall shortly see (cf.\ \refLemma{ZSuppLemma}), we have $\Zed(T)_{\Q} = \emptyset$, and so $\Zed(T)$ is supported in the fibres of finitely many finite primes. We then define the arithmetic degree in this setting to be the Serre intersection multiplicity 
\[ \widehat\deg \ \Zed(T) \ := \ \sum_{\lie p \subset o_k} \ \chi \left( \Zed(T)_{\lie p},  \  \mathcal O_{\Zed(m_1)} \ \otimes_{\mathcal O_{\mathcal M}}^{\mathbb L}  \ \mathcal O_{ \Zed(m_2) }\right) \ \log(N(\lie p))\]
of $\Zed(m_1) \times \Zed(m_2)$ in $\Zed(T)$.  

We begin by describing convenient decompositions of the space $\mathcal M$ and the special cycles, in terms of genera of Hermitian lattices. 
\begin{definition} 
\begin{compactenum}[1.]
\item Let $\mathcal R_d$ denote the set of isomorphism classes of genera $[L]$, where $L$ is a Hermitian $o_k$-lattice such that
\begin{compactenum}[(i)]
\item $V := L \otimes_{\Z} \Q$ is a Hermitian space of signature $(1,1)$;
\item  $L_{\ell} \ := \ L \otimes_{\Z} \Z_{\ell}$ is a self-dual $o_{k, \ell}$-lattice for all $\ell \nmid d$;
\item and for every $\ell | d$, we require that $L_{\ell}$ is an almost self-dual lattice; equivalently, we require that $L_{\ell}$ is a maximal lattice in a non-split Hermitian space of dimension 2 over $k_{\ell}$. 
\end{compactenum}
\item Let $\mathcal R_0$ denote the set of isomorphism classes of genera $[L_0]$, where $L$ is a self-dual Hermitian $o_k$-lattice and $V_0 = L_{0, \Q}$  is of signature $(1,0)$. 
\end{compactenum}

\noindent Here, we consider two genera $[L]$ and $[L']$ to be isomorphic if and only if there are representatives $L \in [L]$ and $L' \in [L']$ such that $L \simeq L'$ as Hermitian $o_k$-modules.

\end{definition}

Suppose $p$ is an odd prime,  let $ \overline \F_p$ denote an algebraic closure of $\F_p$ and fix a trivialization 
\[ \widehat\Z^p(1) \ := \ \prod_{\ell \neq p} \mu_{\ell^{\infty}}(\overline \F_p) \ \simeq  \ \widehat\Z{}^p \]
of the prime-to-$p$ roots of unity over $\overline \F_p$. 
Given a geometric point $\underline A = (A, i_A, \lambda_A) \in \mathrm{M}^d_{(1,1)}(\overline\F_p)$, the prime-to-$p$ Tate module
\[ \mathtt{Ta}^{\, p}(A) \ := \ \prod_{\ell \neq p} \mathtt{Ta}_{\ell}(A) \]
is an $\widehat{o_k}{}^p$-module via the action induced by $i_A$. The polarization $\lambda_A$ determines a Weil pairing
\[ e_{\lambda} \colon \mathtt{Ta}^{\, p}(A) \ \times \ \mathtt{Ta}^{\, p}(A) \ \to \ \widehat\Z^p(1) \ \simeq \ \widehat\Z^p, \]
which in turn induces a Hermitian form $(\cdot, \cdot)_{\lambda_A}$ by the formula 
\[ (x,y)_{\lambda} \ = \ \frac12 \, \left( \,  e_{\lambda}\left(i(\sqrt{\Delta})( x), \,  y \right) \ + \ \sqrt{\Delta}\cdot e_{\lambda}(x,y) \, \right). \]

\begin{lemma} \label{MDecompLemma}
Suppose $p \neq 2$. 
	\begin{compactenum}[(i)]
	\item For every $\underline A \in \mathrm M^d_{(1,1)}(\overline{\F}_p)$, there is a unique genus $[L(A)] \in \mathcal R_d$ such that 
	\[ \Ta^p(A) \ \simeq \ \widehat{ L(A)}{}^p \ := \ L(A) \otimes_{\Z} \widehat{\Z}{}^p \]
	as Hermitian $\widehat {o_k}{}^p$-	modules.
	\item The genus $[L(A)]$ depends only on the connected component of $\mathrm M^d_{(1,1)}$ that contains the point $A$. 
	\item For each $[L] \in \mathcal R_d$, there exists a point $\underline A \in \mathrm M^d_{(1,1)}(\overline{\F}_p)$ with $[L(A)] \simeq [L]$. 
	\end{compactenum}
\begin{proof}

(i) To prove uniqueness, suppose $L, L' \in \mathcal R_d$ with $\widehat{L}{}^p \simeq \widehat{L'}{}^p$. By definition of $\mathcal R_d$, the localizations $L_p$ and $L'_p$ are self-dual lattices when $p \nmid d$, and are identified with the unique maximal lattice in the two-dimensional non-split hermitian space over $k_p$ when $p|d$. In either case, they are isometric, and so $L$ and $L'$ lie in the same genus. 

The existence claim is proved in  \cite[Proposition 2.12]{KRunnglob} when $p$ is ramified, or when $p$ is inert and $p \nmid d$. When $p|d$, the claim follows from \cite[Proposition 3.5]{KRnewunif}, see also the proof of Theorem 6.1 of \emph{loc.\ cit.} The idea is roughly as follows: one shows that  there exists a point $\underline A' \in \mathrm M^d_{(1,1)}(\C)$ such that
\[ \Ta^p(A) \otimes_{\Z} \Q   \ \simeq \ \Ta^p(A') \otimes_{\Z} \Q \]
as Hermitian $k\otimes \A_f^p$-modules. Note that the homology group $V = H_1(A', \Q)$ is a hermitian space of signature $(1,1)$. We may then find a lattice $L  \subset V$ by identifying $L \otimes_{\Z} \widehat\Z{}^p$ with $\Ta^p(A)$, and insisting that $L_p$ is self-dual (resp. almost self-dual) when $p \nmid d$ (resp. $p|d$). One then needs to verify that the genus $[L]$ is independent of all choices, and $[L] \in \mathcal R_d$. 

(ii) This follows immediately from the proof of \cite[Proposition 2.12]{KRunnglob}.

(iii) This follows from a straightforward modification of the proof of \cite[Lemma 5.1]{KRunnglob}. 
\end{proof}

\end{lemma}
Similar assertions hold for the stack $\mathcal E$: for each geometric point $\underline E \in \mathcal E(\overline{\F}_p)$, there is a unique $[L_0(E)] \in \mathcal R_0$ such that 
\[ {\Ta}^p(E) \ \simeq \ \widehat{L_0(E)}{}^p \]
which depends only on the connected component containing $\underline E$,
 and every element of $\mathcal R_0$ appears in this way.

Our next task is to apply the \emph{$p$-adic uniformizations} of Rapoport and Zink, which relate $\M = \mathcal E \times \mathrm{M}_{(1,1)}^d$ and the special cycles to the moduli spaces of $p$-divisible groups that appeared in \refSec{DrinfeldSec}. In the following, we fix a prime $p|d$. Let $\widehat{\mathrm M}$ denote the formal completion of $\mathrm{M}_{(1,1)}^d$ along its fibre at $p$.
By \refLemma{MDecompLemma},  we have a decomposition
\begin{equation} \label{MDecompEqn}
 \widehat{\mathrm M} \ = \ \coprod_{[L] \in \mathcal R_d} \ \widehat{\mathrm M}{}^{ [L]} 
\end{equation}
into components that are characterized by the property that for any geometric point $ \underline A \in \widehat{\mathrm M}{}^{[L]}(\overline{\F}_p)$, there exists an isomorphism
$\mathrm{Ta}^p(A)  \simeq   \widehat{L}{}^p$ of $\widehat{o_k}{}^p$-hermitian modules. 

Fix a lattice $L \in \mathcal R_d$ as above, and set $V = L \otimes_{\Z} \Q$. Let $K^p = \mathrm{Stab}(\widehat{L}{}^p)  \subset U(V)(\A_f^p)$ denote the stabilizer of $\widehat L{}^p$.

Next, we let $V'$ be the hermitian space of dimension two over $k$ whose invariants differ from those of $V$ at exactly $p$ and $\infty$. 
Fixing an isomorphism
\[ V' \otimes_{\Q} \A_f^p \ \simeq \ V \otimes_{\Q} \A_f^p \]
induces an embedding
\[ U(V')(\Q) \ \hookrightarrow  \ U(V')(\A_f^p) \ \simeq  \ U(V)(\A_f^p), \] 
and $U(V')(\Q)$ acts on $U(V)(\A_f^p) / K^p$ by left-multiplication. 

Finally,  fix a geometric point $\underline{\mathbf  A} \in \mathrm M_{(1,1)}^d(\overline{\F}_p)$ with $[L(\mathbf A)] =  [L]$. When $p|d$, the corresponding $p$-divisible group
\[  \underline \X \ := \ \underline {\mathbf A}[p^{\infty}], \]
 together with the induced polarization and $o_{k,p}$-action,  serves as a base point for the Drinfeld upper half-plane $\D$, cf.\ \refDef{DrinfeldUHPDef}.
Let $I $ denote the algebraic group over $\Q$ whose rational points are given by
\[ I(\Q) =  (\End(\mathbf  A, i_{\mathbf A}, \lambda_{\mathbf A}) \otimes \Q)^{\times} . \]
It is a straightforward verification to prove that $I \simeq U(V')$. 
 In particular, $U(V')(\Q)$ acts on $\D$ via the identification
 \[ U(V')(\Q_p) \ \simeq \ (\End(\underline{\X})\otimes_{\Z_p} \Q_p)^{\times}. \]
 
With this notation in place, we obtain the following $p$-adic uniformization theorem, which is a special case of the  the general results of \cite{RZ}; details regarding this particular case can be found in \cite[Theorem 6.11]{KRnewunif}.
\begin{theorem} \label{pUnifMdThm}
Suppose $p|d$ and let $W = W(\overline \F_p)$ denote the ring of Witt vectors. Then there is an isomorphism of formal stacks
\begin{equation} \label{pUnifMd}
 \widehat{\mathrm M}^{[L]} \times_{ o_{k,p}} \Spf W \ \simeq \ \Big[  U(V')(\Q)  \Big\backslash   \D  \ \times \  \left( U(V)(\A_f^p) \big/ K^p \right)\Big].  
\end{equation}\qed
\end{theorem}

Completely analogous results hold for the stack $\mathcal E$: let $L_0 \in \mathcal R_0$ be a lattice, and $\widehat{\mathcal E}^{[L_0]}$ the corresponding component of the formal completion $\widehat{\mathcal E}$, in analogy with \eqref{MDecompEqn}. Fixing a geometric point $\underline{\mathbf  E} = ( \mathbf E, i_{\mathbf E}, \lambda_{\mathbf E}) \in \mathcal E(\F)$, 
we have isomorphisms  
\[ (\End(\mathbf E, i_{\mathbf E}, \lambda_{\mathbf E})\otimes_{\Z} \Q )^{\times} \ \simeq  \ k^{ 1} \ \simeq  \ U(V_0)(\Q),  \]
where $k^{ 1}$ is the group of norm-1 elements of $k$, and $V_0 := L_0 \otimes_{\Z} \Q$.
%

Let $\underline \Y = \underline{\mathbf E}[p^{\infty}]$ denote the  $p$-divisible group attached to $\mathbf E$, which comes equipped with the induced $o_{k,p}$-action and principal polarization, and which serves as the base point for the moduli space $\D_0$, cf.\ \refDef{D0Def};  the reader is reminded that $\D_0 \simeq \Spf W$.

By \cite[Theorem 5.5]{KRunnglob}, and in analogy with \refThm{pUnifMdThm}, 
\begin{equation} \label{pUnifE}
 \widehat{\mathcal E}^{[L_0]}  \times_{o_{k,p}} \Spf W \ \simeq \ \left[ U(V_0)(\Q) \ \Big\backslash \ \D_0 \ \times \ U(V_0)(\A_f^p)/ K_0^p \right]
\end{equation}
where $K_0^p \subset U(V_0)(\A_f^p)$ is the stabilizer of $\widehat{L_0}{}^p$. 

\begin{remark} Fix an isomorphism $V_0 \simeq k$, under which the lattice $L_0$ is identified with a fractional ideal $\lie{a}$, and such that the Hermitian form is given by
\[ (x, y) \ \mapsto \ \frac1{N(\lie a)} \, x \, y' , \qquad x, y \in k; \]
note that $L_0$ determines the class of $\lie a$ in the ideal class group $Cl(k)$. 
Under this identification, we have $U(V_0)(\A_f^p) \simeq (\A_{k,f}^p)^1$, and the stabilizer $K_0^p$ of $\widehat{L_0}{}^p$ is identified with $\widehat{o_k}{}^{p,1}$. 
Recalling that $\D_0 \simeq \Spf W$, 
\[\ \ \ \ \ \ \ \  \ \ \ \ \ \ \   \widehat{\mathcal E}^{[L_0]}  \times_{o_{k,p}} \Spf W \ \simeq \ \left[ k^1 \backslash \A_{k,f}^{p,1} / \widehat{o_k}{}^{p,1} \right] . \ \ \ \ \ \ \ \   \ \ \ \ \ \  \ \ \ \ \ \ \ \ \ \ \ \diamond\]
\end{remark}
 Let $\widehat{\M}$ denote the formal completion of $\M =  \mathcal E \times_{o_k} \mathrm M_{(1,1)}^d $ along its fibre at $p$. We have a decomposition
\[ \widehat{\M}  \ = \ \coprod_{\substack{[L_0] \in \mathcal R_0 \\ [L] \in \mathcal R_d  }} \widehat{\M}^{[L_0], [L]}\]
where $\widehat{\M}^{[L_0], [L]} := \widehat{\mathcal E}^{[L_0]} \times \widehat{\mathrm M}^{[L]} $. Combining \eqref{pUnifMd} and \eqref{pUnifE} yields the $p$-adic uniformization 
{ \small
\[ \widehat{\M}^{[L_0] , [L]} \times \Spf W \ \simeq \ \left[  U(V_0)(\Q) \times U(V')(\Q)  \Big\backslash  \  \left(\D_0   \times U(V_0)(\A_f^p) / K_0^p \right)  \times  \left(\D  \ \times \  U(V)(\A_f^p)/ K^p \right) \right] \]
}

We now turn to the $p$-adic uniformization of the special cycles. Suppose $T \in \mathrm{Herm}_2(o_k)$,  let $\widehat{\Zed}(T)$ be the formal completion along its fibre at $p$, and let
\[ \widehat{ \Zed}(T)^{[L_0], [L]} \ := \ \widehat{ \Zed}(T) \ \times_{\widehat\M} \ \widehat\M^{[L_0], [L]} \]
denote the component corresponding to the pair of genera $([L_0], [L])$. As before, we fix a pair of base points $(\underline{\mathbf E}, \underline{\mathbf A}) \in \widehat\M^{[L_0], [L]}(\F)$, with corresponding $p$-divisible groups $(\underline \Y, \underline \X)$. By considering local invariants, one can check that there is an isomorphism of hermitian spaces
\begin{equation} \label {V'HomEqn} 
V'  \ \simeq  \ \Hom_{o_k}(\mathbf E, \mathbf A)_{ \Q}, 
\end{equation}
where $V'$ is the hermitian space whose invariants differ from those of $V = L_{\Q}$ at exactly $p$ and $\infty$, and the hermitian form on $\Hom_{o_k}(\mathbf E, \mathbf A)_{ \Q}$ is given by \eqref{SpecCycleHermFormEqn}. 

Given an element $x \in V'$, we may use \eqref{V'HomEqn} and take the corresponding completions to obtain elements 
\[ \mathrm{Ta}^p(x)  \ \in \ \Hom_{k \otimes \A_f^p}(\mathrm{Ta}^p(\mathbf E)_{\Q}, \, \mathrm{Ta}^p(\mathbf A)_{\Q}) \qquad \text{and} \qquad x_p \ \in \ \Hom_{o_{k,p}}(\Y, \X)_{ \Q_p} \ =  \ \V . \]

The following proposition is proved in the same way as \cite[Proposition 6.3]{KRunnglob}.
\begin{proposition} \label{padicUnifCycleProp}
 Let $T \in \Herm_2(o_k)$, and as usual, suppose $p|d$. Then there is an isomorphism
\[ \widehat{\Zed} (T){}^{[L_0], [L]}  \times_{o_{k,p}} \Spf W\ \simeq \ \left[ U(V')(\Q) \times U(V_0)(\Q)  \Bigg\backslash  \coprod_{g} \ \coprod_{g_0} \  \coprod_{\mathbf x \in \Omega(T, g, g_0)}  Z(\mathbf x_p) \right] \]
where $g$ and $g_0$ range over $U(V)(\A_f^p) / K^p$ and $U(V_0)(\A_f^p) / K_0^p$ respectively, 
\[ \Omega(T, g, g_0) \ := \ \left\{ \mathbf x = [x_1, x_2] \in ( V')^2 \ | \ (\mathbf x, \mathbf x) = T \text{ and }  g^{-1} \circ \mathrm{Ta}^p(x_i) \circ g_0 \in \Hom(L_0, L) \otimes {\widehat \Z}^p \right\} , \]
and $Z(\mathbf x_p)$ is the local cycle corresponding to $\mathbf x_p = (x_{1,p}, x_{2,p})$ as in \refDef{locCyclesDef}. 
\qed
\end{proposition}

Next, we collect some information regarding the support of a special cycle $\Zed(T)$. 
\begin{lemma} \label{ZSuppLemma}
Suppose $T \in \Herm_2(o_k)$ is positive definite, and set
\[  \Diff(T) := \{ \ell \nmid d \text{ inert}, \ \ord_{\ell} \det T \text{ odd } \} \ \coprod \   \{ \ell \mid d, \ \ord_{\ell} \det T \text{ even } \}.
 \] 
Also, let $V_T$ denote the vector space $k^2$ with Hermitian form given by the matrix $T$.
\begin{compactenum}[(i)]
\item The generic fibre $\Zed(T)_{\Q}$ is empty. 
\item Suppose $\Zed(T)(\overline{\F_v}) \neq \emptyset$. Then for all inert $\ell \neq v$, we have  $\inv_{\ell} V_T = -1$ if $\ell | d$ and $\inv_{\ell} V_T = 1$ if $\ell \nmid d$. 
\item If $\# \mathrm{Diff}(T) > 1$, then $\Zed(T)  = \emptyset$.
\item If $\Diff(T) = \{ \ell \}$ is a single inert prime, then (a) the support of $\Zed(T)$ is contained in the fibre  $ \M_{\ell}$  at $\ell$; and (b) we have a decomposition
\begin{equation} \label{ZedSuppDecompEqn}
 \widehat\Zed(T) = \coprod_{[L_0] \in \mathcal R_0} \coprod_{\substack { [L] \in \mathcal R_d \\ L \otimes \A_f^{\ell} \simeq V_T \otimes \A_f^{\ell}} } \widehat\Zed(T)^{[L_0], [L]}  .
\end{equation}

\end{compactenum}

\begin{proof} We argue as in \cite[Proposition 2.22]{KRunnglob}. Suppose $F$ is an algebraically closed field, and $(\underline E, \underline A, \mathbf x) \in \Zed(T)(F)$ is a geometric point. Since $T$ is non-degenarate, the pair $\mathbf x = [x_1, x_2]$ determines an isomorphism
\[ V_T \ \simeq \ \Hom_{o_k}(E, A)_{\Q} \]
of hermitian spaces. 

If $F = \C$, then we have an embedding
\[ V_T \  \simeq  \ \Hom_{o_k}(E, A)_{\Q} \ \hookrightarrow \Hom_{o_k}( H_1(E, \Q), H_1(A, \Q)) \simeq H_1(A, \Q) ,\]
where $H_1(A, \Q)$ is endowed with the unique Hermitian form $(\cdot, \cdot)$ such that the Riemann form $\la \cdot, \cdot \ra_{\lambda_A}$ induced by $\lambda_A$ satisfies
\[ \left\la i(\sqrt{\Delta}) x, \, y \right\ra_{\lambda_A} \ = \ \frac{1}{2} tr_{k/\Q} (x, \, y). \]
The signature of $H_1(A, \Q)$ is $(1,1)$ by the signature condition \eqref{locCyclesDef}, while the signature of $V_T$ is $(n,0)$ by assumption; hence we obtain a contradiction that proves \textit{(i)}. 

Suppose next that $F$ has characteristic $v>0$, so that
\begin{equation} \label{VTsimEqn}
 V_T(\A_f^v) \ \simeq \  \Hom_{\A_{k,f}^v}(\mathrm{Ta}^v(E)_{\Q},  \ \mathrm{Ta}^v(A)_{\Q} ). 
\end{equation} 
Let $\ell \neq v$ be an inert prime. Then if $\ell \nmid d$, the space $V_{T, \ell}$ contains a self-dual lattice and hence $\inv_{\ell} V_T = 1$; if $\ell | d$, then $V_{T,\ell}$ is 
the non-split hermitian space over $k_{\ell}$, and so $\inv_{\ell} V_T = -1$.  This proves \textit{(ii)}, from which \textit{(iii)} and the statement regarding the support of $\Zed(T) $ in \textit{(iv)} follow easily. 

Finally, it follows from \eqref{VTsimEqn} and the definition of the component $\widehat\Zed(T)^{[L_0], [L]}$  that there is a decomposition
\[ 
 \widehat\Zed(T) = \coprod_{[L_0] , [L] } \widehat\Zed(T)^{[L_0], [L]} 
\]
where the product is over genera $[L_0]$ and $[L]$ in $\mathcal R_0$ and $\mathcal R_d$ respectively, such that
\[ \Hom(L_0, L) \otimes_{\Z} \A_f^p \ \simeq \ V_T \otimes_{\Q}\A_f^p. \]
However, note that there is an isomorphism of $k$-Hermitian spaces
\[ \Hom(L_0, L)_{\Q}  \ \simeq \ L_{\Q}, \]
 since 
\[ \det \left( \Hom(L_0, L)_{\Q} \right) \ = \ \det (L_{0,\Q})^2  \det(L_{\Q}) \ \equiv \ \det(L_{\Q}) \mod N(k^{\times}). \]
Thus
\[  \Hom(L_0, L) \otimes_{\Z} \A_f^p \ \simeq V_T \otimes_{\Q}\A_f^p \ \ \iff \ \  L \otimes_{\Z} \A_f^p \ \simeq V_T \otimes_{\Q}\A_f^p, \]
and the second part of \emph{(iv)} follows immediately. 

\end{proof}
\end{lemma}

Let $T \in \mathrm{Herm}_2(o_k)$ be positive definite with $\Diff(T) = \{p\}$, where $p|d$. By the previous lemma, we may restrict our attention to lattices $L$ such that $L \otimes_{\Z} \A_f^p \simeq V_T \otimes_{\Q} \A_f^p$. 

Viewing $\widehat L^p = L \otimes_{\Z} \widehat\Z{}^p$  as an ad\`elic lattice in $V_T \otimes \A_f^p$, we define
 the Schwarz function 
\[ \varphi'^p_L \in \mathscr S( (V_T\otimes \A_f^p)^2) \ := \ \text{characteristic function of } ( \widehat L{}^p)^2.\]
Let $K'^p \subset U(V_T)(\A_f^p)$ be the stabilizer of $ \widehat{L}^p$, and write
\begin{equation}
 U(V_T)(\A_f^p) \ = \ \coprod_j \ U(V_T)(\Q) \ h_j \ K'^p, \qquad h_j \in U(V_T)(\A_f^p) .
\end{equation}
Finally, we define $\Gamma'_j := U(V_T)(\Q) \cap h_j K'^p h_j^{-1} $. 

\begin{theorem} 
\label{MainThmGeom}
Suppose $T \in \Herm_2(o_k)$ positive definite with $\Diff(T) = \{ p \}$ for $p|d$. Let $[L] \in \mathcal R_d$ with $L \otimes \A_f^p \simeq V_T \otimes \A_f^p$ and fix any $[L_0] \in \mathcal R_0$. Then, with notation as in the previous paragraph,
\[ \widehat\deg \ \Zed(T)^{[L_0], [L]} \ = \  \frac{h(k)}{|o_k^{\times}| \ 2^{o(\Delta) - 1}} \ \cdot \  \mu_p(T)  \ \cdot \ \left(  \sum_{ j}  \ \sum_{\substack{\mathbf  x \in (V_T)^2  \\ h(\mathbf x ) = T \\ \text{mod } \Gamma'_j}} \varphi'^p_L( h_j^{-1}\mathbf x) \right)  \ \cdot \ \log p^2 \]
where $o(\Delta)$ is the number of prime factors of the discriminant $\Delta $ of $k$, and $\mu_p(T)$ is the quantity defined in \eqref{mainLocGeomCorEqn}.

\begin{proof}
We may fix isomorphisms:
	\begin{equation} \label{isomsEqn}
		\Hom(L_0, L) \otimes \A_f^p \ \simeq \  L \otimes \A_f^p \  \simeq \ V_T \otimes \A_f^p
	\end{equation}
In particular, the invariants of $V_T$ differ from those of $V := L \otimes \Q$ at exactly $p$ and $\infty$. Thus $V' = V_T$ in the notation of \refProp{padicUnifCycleProp}, and so we have

 \[ \widehat{\Zed (T)}{}^{[L_0], [L]} \times \Spf W \ \simeq \ \left[ U(V_T)(\Q) \times U(V_0)(\Q)  \Bigg\backslash  \coprod_{g , g_0}  \  \coprod_{\mathbf x \in \Omega(T, g, g_0)}  Z(\mathbf x_p) \right] \]
where for $x \in V_T$, we let $x^p$ denote the image of $x$ in $\Hom(L_0, L) \otimes \A_f^p$ in \eqref{isomsEqn} above, and
\[ \Omega(T, g, g_0) \ := \ \left\{ \mathbf x = [x_1, x_2] \in ( V_T)^2 \ | \ (\mathbf x, \mathbf x) = T \text{ and }  g^{-1} \circ (x_i)^p \circ g_0 \in \Hom(L_0, L) \otimes {\widehat \Z}^p \right\}. \]
 For a given pair $\mathbf x = [x_1, x_2] \in \Omega(T, g, g_0)$, we have already computed the degree of the local cycle $Z(\mathbf x_p)$; indeed, \refCor{mainLocGeomCor} tells us that
 \[ \chi(\Zed(T)_p, \ \mathcal O_{Z(x_{1,p})} \otimes^{\mathbb L} \mathcal O_{Z(x_{2,p})} ) \  =  \  \la Z(x_{1,p}), \ Z(x_{2,p}) \ra  \ = \ \mu_p(T),\]
which in particular depends only on $T$. Thus
\[ \widehat \deg \ \Zed(T)^{[L_0], [L]} \ = \ \mu_p(T) \cdot \# \left[ U(V_T)(\Q) \times U(V_0)(\Q)  \Bigg\backslash  \coprod_{g , g_0}  \  \coprod_{\mathbf x \in \Omega(T, g, g_0)}  \left\{ \mathsf{pt} \right\}  \right] \cdot \log p^2  \]
where on the right, we need to compute the `stack cardinality'. 

Fixing momentarily an element $g_0 \in U(V_0)(\A_f^p) \ = \ (\A_{k,f}^p)^{\times, 1}$, we first compute the cardinality  
\begin{equation} \label{numNogoEqn}
\#  \left[ U(V_T)(\Q) \ \Big\backslash \coprod_{g \in U(V)(\A_f^p) / K^p}  \ \coprod_{x \in \Omega(T, g, g_0)} \ \{ \mathsf{pt} \} \right] . 
\end{equation}
where, as we recall, $K^p \subset U(V)(\A_f^p) $ is the stabilizer of $\widehat L {}^p = L \otimes \widehat\Z{}^p$, and hence is identified with $K'^p \subset U(V_T)(\A_f^p)$. 
Without loss of generality, we may normalize \eqref{isomsEqn} so that ${\widehat L}{}^p$ is identified with
\[ g_0 \cdot \Hom(L_0, L) \otimes_{\Z} \widehat\Z{}^p, \]
 and so the quantity in \eqref{numNogoEqn} then becomes
\begin{equation} \label{numNoGoEqn2}
 \# \left[ U(V_T)(\Q) \ \Big\backslash \coprod_{g \in U(V_T)(\A_f^p) / K'^p}  \ 
	\left\{ \mathbf x \in (V_T)^2 \ | \ (\mathbf x, \mathbf x) = T \ \text{and} \ x \in g \cdot \widehat {L'}{}^p \right\}  \right]
\end{equation}
Since $T$ is non-degenerate, the stabilizer of any $\mathbf x$ appearing in \eqref{numNoGoEqn2} is trivial. Thus
\[ \eqref{numNoGoEqn2} \ = \ \sum_j \ \sum_{ \substack{ \mathbf x \in (V_T)^2 \\ (\mathbf x, \mathbf x) = T \\ \text{mod} \ \Gamma'_j }} \  \varphi'^p_{L}(  h_j^{-1} \mathbf x); \]
we note that this quantity is independent of $g_0$, $L_0$, the choice of representative $L$ in its genus $[L]$, and the choices of isomorphisms in \eqref{isomsEqn}.

On the other hand, we observe that
\[  U(V_0)(\Q) \big\backslash U(V_0)(\A_f^p) \big/ K_0^p \  \simeq  \ k^{\times, 1} \big\backslash ( \A_{k,f}^{p})^1  \big/ \widehat{o_k}^{p, 1},  \]
and the latter double quotient is easily seen to be isomorphic to $Cl(k)^2$, where $Cl(k)$ is the class group. Since the 2-torsion in $Cl(k)$ has order $2^{o(\Delta) - 1}$, the order of $Cl(k)^2$ is $h(k) / 2^{o(\Delta) - 1}$, and so (taking automorphisms into account)
\[ \# \left[ U(V_0)(\Q) \backslash U(V_0)(\A_f^p) / K^p \right] \  =  \  \# \left[ k^{\times, 1} \backslash  \A_{k,f}^{p,1}  / \widehat{o_k}^{p, 1} \right] \ = \ \frac{h(k)}{|o_k^{\times}|  \ 2^{o(\Delta) -1 } }. \]
%
Combining these calculations yields the proposition. 

\end{proof}
\end{theorem}

\subsection{Non-degenerate Fourier coefficients of Siegel Eisenstein series}  \label{EisSect}
{\ \\ }

In this section, we recall some general definitions and formulas for the Eisenstein series attached to unitary groups that are of interest for our main theorem.  Let $G = U(n,n)$ be the quasi-split unitary group, and $P$ be the standard Siegel parabolic; we view both  as algebraic groups defined over $\Q$. 

Fix a multiplicative character $\eta \colon \A_{k}^{\times} \to \C^{\times}$ whose restriction to $\A_{ \Q}^{\times}$ corresponds to the quadratic character $\chi_k$ attached to the field extension $k / \Q$. We also let $\psi \colon \A = \A_{\Q} \to \C^{\times}$ denote the standard additive character that is trivial on $ \widehat \Z $ and $\Q$. 

Central to our investigation is a family of {degenerate principal series representations}, parametrized by a variable $s \in \C$, and defined via smooth normalized induction:
\[ I(s, \eta) \ := \   \mathrm{Ind}_{P( \A)}^{G(\A)}  \ (\eta \circ \det) |\det|_{\A}^{s} \]
A \emph{holomorphic section} is a family of vectors $\Phi(s, \cdot) \in I(s, \eta)$ parametrized by $s \in \C$ such that for each $g \in G(\A)$, the assignment
\[ s \ \mapsto \Phi(s, g) \] 
 is holomorphic as a function of $s$; for such a  section $\Phi$, we  form the \emph{Eisenstein series}
\begin{equation}
 E(g, s, \Phi) :=  \sum_{ \gamma \in P(\Q) \backslash G(\Q) } \ \Phi(s, \gamma g),
\end{equation}
at least for $Re(s)$ sufficiently large.

The aim of this section is to collect information about the Fourier coefficient $E_T(g, s, \Phi)$ for a non-degenerate matrix $T \in \mathrm{Herm}_{n}(o_k)$. When $Re(s)$ is sufficiently large and $\Phi= \otimes \Phi_v$ is factorizable, there is a product expansion
\begin{equation} \label{ECoeffProdEqn} E_T(g, s, \Phi) \ = \ \prod_{v \leq \infty} \ W_{T, v}(g_v, s, \Phi_v) 
\end{equation}
in terms of the \emph{local Whittaker functions}
\[ W_{T, v}(g_v, s, \Phi_v) \ := \ \int_{\mathrm{Herm}_n(k_v)} \ \Phi_v \left( (\begin{smallmatrix} & 1_n \\ -1_n & \end{smallmatrix}) \cdot b \cdot g_v, \ s  \right) \ \psi_v( - tr( T b)) \ \mathrm d  b,\]
where $\mathrm db$ is the additive Haar measure on $\Herm_n(k_v)$ normalized to be self-dual with respect to the pairing $(n_1, n_2) \mapsto \psi_v(tr(n_1 n_2))$.
Each Whittaker function is entire in $s$. Furthermore, for any sufficiently large finite set $\Sigma$ of places, there is an $L$-function $L^{\Sigma}(s)$ such that the expression
\begin{equation} \label{ECoeffContEqn}
 \mathbf E_T(g, s, \Phi) \ = \ L^{\Sigma}(s)^{-1} \cdot \prod_{ v \in \Sigma} \ W_{T, v}(g_v, s, \Phi_v) 
\end{equation}
furnishes a meromorphic continuation of \eqref{ECoeffProdEqn} to all $s \in \C$ that is holomorphic for $Re(s) \geq 0$, cf.\ \cite[\S8]{KRunnglob} or \cite{Tan}. 

Note that any standard section $\Phi$ is determined by its value $\Phi(0, \cdot) \in I(0, \eta)$, and so we would like to understand this latter space as a representation of $G(\A)$. We begin by describing the local case: let $V_v$ be an $n$-dimensional hermitian space over the local field $k_v$, for some  place $v$, and let $\mathcal S(V_v^n)$ denote the space of Schwarz functions on $V_v^n$. We obtain a map
\[ R_v \colon \mathcal S(V_v^n) \ \to \ I_v(0, \eta_v) = \mathrm{Ind}_{P(k_v)}^{G(k_v)} \left( \eta_v ( \det) \right) \ \qquad  \text{ where  }  \ R_v(\varphi_v) (  g_v)  \ =  \  \left(\omega_{v}( g_v) \varphi_v\right)(0) \] 
and $\omega_v$ is the local Weil representation. If we let $R_v(V_v)$ denote the image of this map, we obtain a decomposition
\[ I_v(0, \eta_v) \ = \ \bigoplus_{V_v} R_v( V_v) \]
into irreducible components, cf.\ \cite{KudlaSweet}; here the sum is over isomorphism classes of hermitian space $V_v$ of dimension $n$, and so there are either 1, 2, or $n+1$ summands, corresponding to the cases $v$ split, $v$ non-split, and $v = \infty$ respectively. Given a Schwarz function $\varphi_v \in \mathcal S(V_v^n)$, we let $\Phi_{\varphi_v}(s)$ denote the unique standard section of $I_v(s, \eta_v)$ such that $\Phi_{\varphi_v}(0, \cdot) = R_v(\varphi_v)$; the section $\Phi_{\varphi_v}$ is called the \emph{Siegel-Weil standard section} attached to $\varphi_v$.  

\begin{proposition}[{\cite[Proposition 10.1]{KRunnglob}}] \label{WhittakerProp}
Suppose $v$ is a finite prime, $\varphi_v$ is the characteristic function of an $o_{k,v}$ hermitian lattice $L_v$ of rank $n$ and $\Phi_v(s)$ is the associated Siegel-Weil standard section. 
Then for $r \in \Z_{\geq 0}$,
\[ W_{T,v}(e, r, \Phi_v) \ = \ \gamma_v(V_v)^n \ | N(\det \, S )|_v^{n/2} \ |\Delta|_v^{e}  \ \alpha_v(S_r, T) \]
where 
\begin{compactenum} [(i)]
\item $S$ is any matrix representing the hermitian form on $L_v$;
\item $\alpha_v(S_r,T)$ is the representation density as in \eqref{repDenDefEqn};
\item $\Delta$ is the discriminant of $k$ and $e := \frac14 n(3n+4r-1)$;
\item and $\gamma_v(V_v)$ is an eighth root of unity depending only on $V_v = L_v \otimes_{\Z} \Q$, see \cite[Equation 10.3]{KRunnglob}. 
\end{compactenum}
In particular,
\[ W'_{T,v}(e, 0, \Phi_v) \ =  \   \gamma_v(V_v)^n \ | N(\det \, S)|_v^{n/2} \ |\Delta|^e_v\  \alpha_v'(S,T) \cdot \log v .\]
\qed
\end{proposition}

Passing to the global picture, we observe that $ I(0, \eta)  \ = \ \otimes' \ I_v(0, \eta_v) $ can be decomposed as a restricted tensor product; a  pure tensor $\Phi = \otimes \Phi_v$ lies in $I(0, \eta)$ if and only if for almost all $v$, the local component $\Phi_v$ is the indicator function of $G(\Z_v)$. As a consequence, we may write
\[ I(0, \eta) \ = \  \bigoplus_{\mathcal C} R(\mathcal C) \]
where the sum on $\mathcal C$ is over isomorphism classes of Hermitian spaces over $\A_k$ of rank $n$, such that $\inv_{v} ( \mathcal C \otimes_{\A_k} k_{v} ) = 1 $ for almost all $v$. If  $\prod_v \inv_v( \mathcal C_v) = 1$, then there exists a hermitian space $\mathcal V$ over $k$ such that $\mathcal C  = \mathcal V \otimes_k \A_k$, and in this case we say that $\mathcal C$ is \emph{coherent}. 
%

\begin{theorem}[Extended Siegel-Weil formula] \label{SWThm}
Suppose that $\mathcal V$ is an $n$-dimensional positive definite hermitian space  over $k$, and let $H = U(\mathcal V)$. If $\varphi \in \mathscr S( \mathcal V(\A)^n)$ is an adelic Schwarz function, we define
\begin{equation} \label{thetaDefEqn}
\mathbf I( g, \varphi) \ : = \ \int\limits_{H(\Q) \backslash H(\A)}  \ \left( \sum_{x \in \mathcal V(\Q)^n} \omega(g)   \ \varphi( h^{-1} x) \right) \ \mathrm dh, \qquad g \in G(\A). 
\end{equation}
Here the measure $\mathrm dh$ is the Haar measure normalized so that $\mathrm{vol}( H(\Q) \backslash H(\A) , \mathrm dh) = 1$. Then
\[ \mathbf E(g, 0 , \Phi) \ = \ 2 \ \mathbf I ( g , \varphi). \] 
 \qed
\end{theorem}

%

%

Specializing to the case $n=2$, we now describe the Eisenstein series that figures in our main theorem. Let $L \in \mathcal R_d$ be a lattice. We define a (non-standard) section $\Phi^*_L = \otimes \Phi^*_{L, \ell}$ as follows:
\begin{itemize}
	\item If $\ell \nmid d$ is a finite prime,  set $\Phi^*_{L, \ell}(s)$ to be the standard Siegel-Weil section attached to the characteristic function $\varphi_{L,\ell}$ of $(L_{\ell})^2$.
	\item If $\ell = \infty$,  take $\Phi^*_{L, \infty}(s)$ to be the Siegel-Weil section attached to the standard Gaussian on the positive definite Hermitian space $(\C^n)^2$.
	\item Suppose $\ell | d$. Let $\Phi_{L, \ell}$ denote the Siegel-Weil section attached to the characteristic function of $L_{\ell}$, and recall that $\inv_{\ell} V_{\ell} = -1$. Let $V^{+}_{\ell}$ be the Hermitian space with $\inv_{\ell} V^{+}_{\ell} = 1$, and fix a self-dual lattice $L_{\ell}^+$ inside it. Denote its characteristic function by $\varphi_{L^+_{\ell}}$ and let $\Phi_{L^+_{\ell}}(s)$ be the corresponding Siegel-Weil section. We then define
	\[ \Phi^*_{\ell}(s)  \ := \  \Phi_{L_{\ell}}(s) \ + \  A_{\ell}(s)  \ \Phi_{L^+_{\ell}}(s) \]
	where\footnote{
	It will turn out that for our purposes, only the values $A_{\ell}(0) = 0$ and $A'_{\ell}(0) = \frac{\log \ell}{1- \ell^2 }$ will be relevant, and so one may instead take any other function that yields the same values when it and its derivative are evaluated at $s=0$; we have chosen this particular function only for concreteness' sake. 
} 
	\[ A_{\ell}(s) \ := \ \frac{1}{2(1-\ell^2  ) } \ \left( \ell^s - \ell^{-s } \right). \]
\end{itemize}
Next, let
\[\mathbf E(g, s, [L]) \ := \ \mathbf E(g, s, \Phi^*_L)\]
 be the corresponding Eisenstein series, which only depends on the genus $[L]$. We also consider the `classicalized' Eisenstein series, as follows. Let
 \[ \lie{H}_2 \ := \ \left\{ z \in M_2(\C) \ | \ v(z) := \frac{1}{2i}\left( z - {}^t\bar z \right) > 0 \right\} \]
 denote the Hermitian upper half-space; for $z \in \lie{H}_2$, we may write $v(z) = a \cdot {}^t\bar a$ for some $a \in GL_2(\C)$, and we set $u(z) := \frac{1}{2} (\tau + {}^t\bar \tau)$. Define elements 
 \[ g_{z, \infty} \ = \ \begin{pmatrix} \Id_2 & u(z) \\ & \Id_2 \end{pmatrix} \begin{pmatrix} a & \\ & {}^t \bar{a}^{-1} \end{pmatrix} \ \in \ G(\R), \qquad \text{and} \qquad  g_z = (g_{z, \infty}, 1, 1, \dots) \in G(\A) , \]
 and set
  \[ \mathcal E(z, s, [L]) \ :=  \ \eta_{\infty}(\det \, a)^{-1}  \det(v)^{-1} \ \mathbf E(g_z, s, [L]). \]
This normalization ensures that $\mathcal E(z, s, [L])$ transforms as a Hermitian modular form of weight $2$ on $\lie{H}_2$. 
%

Let $T \in \mathrm{Herm}_2(o_k)$ be a positive definite matrix. As was the case in the previous section, it turns out that we are most interested in genera $[L] \in \mathcal R_d$ such that $L \otimes_{\Z} \A_f^p \simeq V_T \otimes_{\Q} \A_f^p$; we fix such a lattice $L$ and an isomorphism in the following discussion. 


Let $V := L \otimes \Q$; i.e.\ $V$ is, up to isomorphism, the unique hermitian space of signature (1,1) whose invariants differ from those of $V_T$ at exactly $p$ and $\infty$. Fix a non-zero element $\mathtt b \in \bigwedge^{4}_{\Q} \mathrm{Herm}_2(k)^*$, and non-zero elements 
\[  \mathtt a_V \in \bigwedge\nolimits^{8}_{\Q} (V^2)^* \qquad \text{and} \qquad   \mathtt a_{V_T} \in \bigwedge\nolimits^{8}_{\Q} (V_T^2)^*. \] 
 As explained in e.g.\  \cite[\S 10]{KRunnglob}, these elements determine  gauge forms $\omega_V = \omega( \mathtt a_V, \mathtt b)$ on $V$ and $\omega_{V_T}  = \omega(\mathtt a_{V_T}, \mathtt b)$ on $V_T$ , which in turn  induce factorizations 
\begin{equation} \label{matchTamEqn}
 \mathrm d h \ = \ \frac12  L(1, \chi_k)^{-1} \ \prod_{v} \ \mathrm d_v h \qquad \text{and} \qquad \mathrm d h_T \ = \ \frac12  L(1, \chi_k)^{-1} \ \prod_{v} \ \mathrm d_v h_T 
\end{equation}
of the Haar measures $\mathrm dh $ and $\mathrm d h_T$ on $U(V)(\A)$ and $U(V_T)(\A)$ respectively in terms of Tamagawa measures; for example, each term $\mathrm d_v h$ is a measure on $U(V)(\Q_v)$ determined by $\omega_V$. We say that the gauge forms (or the respective decompositions of Haar measures) are \emph{matched} if $\omega_{V_T}  = \gamma^* \omega_V$ for some isomorphism $\gamma \colon V_T \otimes_{\Q} \overline \Q \isomto V \otimes_{\Q} \overline \Q$.

We briefly recall the notation that we had set up in the previous section: we view $\widehat L^p$ as an adelic lattice in $V_T \otimes \A_f^p$, and define a Schwarz function
\begin{equation}
	\varphi'^p_L \ \in \ \mathscr S\left( (V_T \otimes \A^p_f)^2 \right)  
\end{equation}
as the characteristic function of $( \widehat L^p)^2$. Put
\[ K'^p  := \  \mathrm{Stab}_{U(V_T)(\A^p_f)}(  \widehat L^p) \ \subset \ U(V_T)(\A^p_f), \]
and writing
\begin{equation} \label{HTCosetsEqn}
 U(V_T)(\A_f^p) \ = \ \coprod_j \ U(V_T)(\Q) \ h_j \ K'^p, 
\end{equation}
set  $\Gamma'_j \ = \ h_j K'^p h_j^{-1} \cap U(V_T)(\Q)$. 

\begin{theorem}  \label{MainThmEis}
Let $[L] \in \mathcal R_d$, and suppose $T \in \mathrm{Herm}_2(o_k)$ is positive definite. Set
\[  \Diff(T) \ := \ \{ \ell \nmid d \text{ inert}, \ \ord_{\ell} \det T \text{ odd } \} \ \bigcup \  \{ \ell \mid d, \ \ord_{\ell} \det T \text{ even } \}.
 \] 
 \begin{enumerate}[(i)]
\item Suppose $p \in \Diff(T)$, but  $L \otimes_{\Z} \A_f^p \not\simeq V_T \otimes_{\Q} \A_f^p$. Then $\mathcal E_T'(z, 0, [L]) = 0$. Moreover, this is the case whenever $\# \Diff(T) \geq 2$. 
\item If  $\Diff(T) = \left\{ p \right\}$ is an inert prime with $p|d$, and $ L \otimes_{\Z} \A_f^p \simeq V_T \otimes_{\Q} \A_f^p$, then
\begin{equation} \label{E'ThmEqn}
 \mathcal E_T'(z, 0, [L]) \  = \ C_{[L]}  \ \sum_{j} \ 
  \sum_{ \substack{\mathbf x \in V_T^2 \\ (\mathbf x, \mathbf x) = T \\ \text{mod } \Gamma'_j }} \varphi'^p_{L}(h_j^{-1} \mathbf x)  \ \cdot \ \mu_p(T) \  \log(p) \ q^T 
 \end{equation}
  where $q^T := e^{ 2 \pi i \, \mathrm{Tr}( T z) } $ and
  \[ C_{[L]}  = \ L(1, \chi_k)^{-1} \ \mathrm{vol} \left( U(V_T)(\R), \, \mathrm{d}_{\infty} h_T\right) \ \mathrm{vol}\left( K_L, \, \mathrm d^{\infty} h \right).  \]
  Here $K_L \subset U(V)(\A_f)$ is the stabilizer of $ \, \widehat L = L \otimes \widehat \Z$, and the measures  $\mathrm{d}_{\infty} h_T$ and $ \mathrm{d}^{\infty} h \ = \ \prod_{v \neq \infty} \mathrm d_v h$ are components of a matched decomposition as in \eqref{matchTamEqn}. 
Moreover, the constant $C_{[L]}$ is independent of $T$ and all choices appearing in the decompositions of Haar measures.

\end{enumerate}
\begin{proof}
Taking $g = g_z \in G(\A)$ as above, let $\Sigma$ be a sufficiently large finite set of primes containing all the primes in $\Diff(T)$ so that, upon taking the derivative in \eqref{ECoeffContEqn}, we obtain 
{ \small
\begin{equation}\label{derivEEqn}
 \mathbf E_T'(g, 0, [L]) \ = \ -\frac{L^{\Sigma,'}(0)}{L^{\Sigma}(0)^2} \prod_{v \in \Sigma} W_{T, v}(g_{v}, 0, \Phi^*_v) \ + \ L^{\Sigma}(0)^{-1} \cdot \left( \sum_{v \in \Sigma} W'_{T, v}(g_v, 0, \Phi^*_v) \prod_{\substack{v' \in \Sigma \\ v' \neq v}} W_{T, v'}(g_{v'}, 0, \Phi^*_{v'}) \right) 
\end{equation}
}
for $[L] \in \mathcal R_d$ and $\Phi^* = \Phi^*_L$. If $v$ is a finite prime such that $L \otimes_{\Z} \Q_{v} \not\simeq V_T \otimes_{\Q} \Q_{v}$, then  $L_v$ does not represent $T$, and so the representation density $\alpha_v(S, T)$ vanishes, where $S$ is any matrix representing the hermitian form on $L$. By  \refProp{WhittakerProp},
\[W_{T,v}(g_v, 0, \Phi^*_v) \ =\   W_{T, v}(e, 0, \Phi^*_v) \ = \ 0 \] 
as well. Thus if there are at least two such primes, then $\mathbf E_T'(g, 0, [L]) = 0$ as each term in \eqref{derivEEqn} vanishes. 

By the definition of $\mathcal R_d$, we have that $L\otimes_{\Z} \Q_{\ell} \not\simeq V_T \otimes_{\Q} \Q_{\ell}$ for any prime $\ell \in \Diff(T)$, since $\inv_{\ell} (L\otimes \Q )= 1$ when $\ell \nmid d$ and $\inv_{\ell}( L \otimes \Q) = -1$ when $\ell \mid d$. This proves the first statement. 
%
%

Thus from this point on, we suppose $\Diff(T) = \{ p\}$ for some $p|d$,  and we fix a lattice $[L] \in \mathcal R_d$  with an isomorphism $L \otimes_{\Z} \A_f^p \simeq V_T \otimes_{\Q} \A_f^p$. 
Then \eqref{derivEEqn} gives
\begin{equation}   \label{prodDecompEqn}
 \mathbf E_T'(g, 0, [L]) \ = \ W'_{T,p}(e, 0, \Phi^*_p) \ \cdot\  L^{\Sigma}(0)^{-1} \ \prod_{\substack{v \in \Sigma \\ v \neq p}} W_{T, v}(g_{v}, 0, \Phi^*_{v})  
\end{equation}
By using  \refProp{WhittakerProp} and the definition of $\Phi^*_p$, 
\begin{align*} 
	 W'_{T, p}(e,0, \Phi^*_p) \ =&   \ \gamma_p(V_p^-)^2 \   |p^2|_p \ \alpha'_p \left( (\begin{smallmatrix} 1 & \\ & p \end{smallmatrix} ), T \right)  \log p   \
	 + \   \gamma_p(V^+_p)^2 \ A'_p(0) \ \alpha_p \left( (\begin{smallmatrix} 1 & \\ & 1 \end{smallmatrix} ), T \right)   \\
	 =& \ \log p \cdot \left( \frac{ \gamma_p(V_p^-)^2 }{p^2}  \alpha'_p \left( (\begin{smallmatrix} 1 & \\ & p \end{smallmatrix} ), T \right) \ + \  \frac{\gamma_p(V^+_p)^2 }{1-p^2 } \alpha_p \left( (\begin{smallmatrix} 1 & \\ & 1 \end{smallmatrix} ), T \right)  \right)
\end{align*}
It is easily seen that $\gamma_p(V^+_p) = -  \gamma_p(V^-_p)$, cf.\ \cite[Equation (10.3)]{KRunnglob}. By \refCor{repDenFinalCor}, 
\begin{equation} \label{W'FinalEqn}
	 W'_{T,p}(e, 0, \Phi^*_{L,p}) \ = \  \gamma_p(V^+_p)^2 \cdot \frac{(p+1)^2}{p^3}  \cdot \mu_p(T) \cdot \log p.
\end{equation}
We may compute the remaining terms in \eqref{prodDecompEqn} via the Siegel-Weil formula. 
Let 
\begin{equation} \label{lieLDefEqn}
\lie{L} \ \subset \ V_T
\end{equation}
be a lattice such that $\lie{L} \otimes \widehat\Z{}^p $ is identified with $L \otimes \widehat\Z{}^p$ and $\lie{L}_p$ is self-dual. We then have a Schwarz function
\[ \varphi'   \ = \ \varphi'_{\infty} \otimes \varphi'_{f} \ \in \  \mathscr S(V_T(\A)^2), \]
 where $\varphi'_f $ is the indicator function of $ (\lie{L}\otimes \widehat \Z)^2$ in $(V_T \otimes \A_f)^2$, and $\varphi'_{\infty}$ is the standard Gaussian on the positive definite space $(V_T \otimes \R)^2 \simeq \C^4$. Note that $\varphi'^p_f = \varphi'^p_L$ as before. 
 
 Let $\Phi' = \otimes \Phi'_v$ the corresponding standard Siegel-Weil section. It follows immediately from definitions that
\[ W_{T,v}(g_v, 0, \Phi'_v) \ = \ W_{T,v}(g_v, 0, \Phi^*_v) \qquad \text{for all } v \neq p. \]
Combining this fact with the extended Siegel-Weil formula (\refThm{SWThm}) and the product expansion \eqref{ECoeffContEqn}, we have  for any $g \in G(\A)$, 
\[ 2 \, \mathbf I_T(g, \varphi') \ = \ \mathbf E_T(g, 0 , \Phi') \ = \ L^{\Sigma}(0)^{-1} \ \left( \prod_{\substack{ v \in \Sigma \\ v \neq p}} W_{T,v}(g_v, 0, \Phi^*_v) \right)\ W_{T,p}(g_p, 0, \Phi'_p)  \]
where
\[ \mathbf  I_T(g, \varphi') \ = \ \int\limits_{H_T(\Q) \backslash H_T(\A)} \ \left( \sum_{ \substack{ \mathbf x \in \mathcal (V_T)^2 \\ (\mathbf x, \mathbf x) = T}} \omega(g)   \ \varphi'\left( h_T^{-1} x\right) \right) \ \mathrm dh_T \]
is the $T$'th Fourier coefficient of the theta integral \eqref{thetaDefEqn} and for ease of notation, we have abbreviated $U(V_T) = H_T$. 
Writing
\[ H_T(\A) \ = \ \coprod_j \ H_T(\Q)  \cdot h_j \cdot K'^p \cdot H_T(\R) \cdot H_T(\Q_p) \]
for a collection of elements $h_j \in H_T(\A_f^p)$ appearing in \eqref{HTCosetsEqn}, we  then have
\[ \mathbf  I_T(g, \varphi_T) \ = \  \sum_j \ \sum_{ \substack{ \mathbf x \in \mathcal (V_T)^2 \\ (\mathbf x, \mathbf x) = T \\ \text{mod } \Gamma'_j} } \ \left[ \ \int\limits_{K'^p \, H_T(\R)\, H_T(\Q_p)} \ \omega(g)   \ \varphi' \left( h_T^{-1}  h_j^{-1} \mathbf x \right) \ \mathrm d h_T \right]. \]
We again specialize to the case 
\[ g \ =  \ g_z \ = \ (g_{z, \infty}, 1, 1, \dots) \ \in \ G(\A) .\]
Recall that we had chosen matched factorizations
\[ \mathrm d \, h_T \ = \ \frac12 L(1, \chi_k)^{-1} \ \prod_v \mathrm d_v h_T \qquad \text{and} \qquad   \mathrm d \, h \ = \ \frac12 L(1, \chi_k)^{-1} \ \prod_v  \mathrm d_v h \]
for the Haar measures $\mathrm d h_T$ and $ \mathrm d  h$ on $H_T(\A)$ and $H(\A)$ respectively. 
If we set $\mathrm d^{p, \infty} \, h_T = \prod_{v \neq p, \infty} \mathrm d_v h_T$, then by definition of $K'^p$ and $\varphi'$, 
\[ \int_{K'^p}  \ \varphi'^p_f(h^{-1} h_j^{-1} \mathbf x) \ \mathrm d^{p, \infty}h  \ = \ \varphi'^p_f(h_j^{-1} \mathbf x)  \ \ \mathrm{vol}(K'^p, \, \mathrm d^{p, \infty} h_T) \ = \ \varphi'^p_L(h_j^{-1} \mathbf x)  \ \ \mathrm{vol}(K_L^p, \, \mathrm d^{p, \infty} h),   \]
where we use part \textit{(i)} of \refLemma{volLemma} below in the second equality. 
On the other hand, a straightforward computation using the archimedean Weil representation, cf.\ \cite[Equation (7.4)]{KRunnglob}, yields 
\[ \int_{H_T(\R)} \omega(g_{z, \infty}) \  \varphi'_{\infty}(h_T^{-1} h_j^{-1} \mathbf x)  \ \mathrm d_{\infty} h_T \ = \ \mathrm{vol}(H_T(\R), \mathrm d_{\infty}h_T) \ \eta_{\infty}(\det{a})  \ \det(v) \  q^T. \]
Combining these calculations with part \textit{(ii)} of \refLemma{volLemma} below gives
\begin{equation} \label{ITFinalEqn}
 \mathbf I_T(g_z, \varphi') \ = \ C \ \sum_j    \ \left( \sum_{ \substack{ \mathbf x \in \mathcal (V_T)^2 \\ (\mathbf x, \mathbf x) = T \\ \text{mod } \Gamma'_j}}  \ \varphi_f'^p(\mathbf x)\right) \eta_{\infty}(\det{a})  \ \det(v) \ W_{T,p}(e, 0, \Phi_p') \  q^T 
\end{equation}
where 
\begin{align*}
 C \  =& \ \frac12 \ L(1, \chi_k)^{-1} \ \mathrm{vol}\left(H_T(\R), d_{\infty} h_T \right) \ \mathrm{vol} \left(K_L, d^{\infty} h \right)  .
\end{align*}
The proof of \cite[Lemma 9.5]{KRunnglob} implies that $C$ is independent of $T$ as well as the choices involved in the decompositions of Haar measures; indeed this constant can be written as a ratio of volumes for which the independence is immediately evident.  

Using \eqref{W'FinalEqn} and \eqref{ITFinalEqn}, 
\begin{align*}
\mathbf E'(g_z, 0, [L]) \ =& \ W_{T,p}'(e, 0, \Phi^*_p) \ \frac{2 \, \mathbf I_T (g_z, \varphi')}{W_{T,p}(e, 0, \Phi_p')}  \\
=&    \  2 \  C  \cdot \mu_p(T) \cdot \left( \sum_{ \mathbf x} \varphi'^p_f(\mathbf x) \right) \cdot \eta_{\infty}(\det{a})  \cdot \det(v) \cdot  \log p \cdot q^T. 
\end{align*}
The result now  follows from rewriting the above in terms of the classicalized Eisenstein series $\mathcal E(z, 0, [L])$.
\end{proof}
\end{theorem}

It remains to prove the following lemma:
\begin{lemma}  \label{volLemma}
With the notation as in the previous theorem, we have
	\begin{enumerate}[(i)]
	\item  $\vol(K_L^p, d^{p, \infty} h) = \vol(K'^p, d^{p, \infty} h_T) $;

	\item and 
	\[ \int_{H_T(\Q_p)} \ \varphi' \left( h_T^{-1} \mathbf x \right) \mathrm d_p \, h_T \ = \ \gamma_p(\V_p^+)^{-2}  \ \frac{p^3}{(p+1)^2} \ \mathrm{vol}(K_{L,p},  \, \mathrm d_p h) \ \cdot \ W_{T,p}(e, 0, \Phi_p'), \]
	where $\mathbf x \in (V_T \otimes \Q_p)^2$ with $(\mathbf x , \mathbf x ) = T$. 
	\end{enumerate}
\begin{proof} We use the following expression, found in \cite[Lemma 10.4]{KRunnglob}, for the volume of  the stabilizer $K_v = K_{L,v}$ of the localized lattice $L_v$ at a finite place $v$. Fix a basis $\mathbf e = \{ e_1, e_2\}$ for ${L}_v$, with  $S = (\mathbf e, \mathbf e) $ the corresponding matrix of inner products. We also fix a $\Z_v$-basis $ \mathbf f = \{f_1, f_2, f_3, f_4 \}$ for $o_{k,v}^2$,  so that 
\[ \mathbf e \otimes \mathbf f \ := \ \left\{  e_i \otimes f_j\ | \  i = 1,2, \, j = 1, \dots, 4  \right\} \]
 is a $\Z_v$-basis for $L_v^2$. Finally,  fix a $\Z_v$-basis $\mathbf c$ for $\mathrm{Herm}_2(o_{k,v})$ whose span is a self-dual lattice.  Then
\begin{equation} \label{volKLEqn}
	\mathrm{vol} ( K_{v}, \ \mathrm{d}_v h ) \ = \ L_v(1, \chi_k) \ \frac{ | \mathtt a( \mathbf e \otimes \mathbf f)|_v}{|\mathtt b(\mathbf c)|_v}  \ |\Delta|_v \ \alpha_v(S, S),
\end{equation}
where $\mathtt a \in \bigwedge_{\Q}^{8}(V^2)^*$ and $\mathtt b \in \bigwedge_{\Q}^4 \mathrm{Herm}_2(o_k)^*$ were the fixed non-zero vectors used to construct the local measure $\mathrm d_v h$ on $H(\Q_v) = U(V)(\Q_v)$. 
Similarly, if $K' \subset H_T(\A_f)$ is the stabilizer of the lattice $\lie{L} \subset V_T$ as in \eqref{lieLDefEqn}, then
\[ 	\mathrm{vol} ( K'_{v}, \, \mathrm{d}_v h_T ) \ = \ L_v(1, \chi_k) \ \frac{ | \mathtt a'( \mathbf e \otimes \mathbf f)|_v}{|\mathtt b(\mathbf c)|_v}  \ |\Delta|_v \ \alpha_v(S', S') \]
where $\mathtt a' = \gamma^* \mathtt a$ is the pullback under an isomorphism $\gamma\colon V_T \otimes \overline\Q \isomto V \otimes \overline \Q$ and $S' = (\mathbf e', \mathbf e')$ is the matrix of inner products of any basis $\mathbf e'$ of $L'_v$.  Furthermore, Kudla and Rapoport compute
\begin{equation} \label{compAEqn}
 \frac{ |\mathtt a' (\mathbf e' \otimes f)  |_v}{|\mathtt a (\mathbf e \otimes f) |_v }  \ = \ \frac{|\det \, S'|_v^2}{ |\det \, S|_v^2}
\end{equation}
and then conclude in \cite[Lemma 10.4]{KRunnglob} that
\begin{equation} 
 \frac{\mathrm{vol} ( K'_{v}, \mathrm{d}_v h_T ) }{\mathrm{vol} ( K_{v}, \mathrm{d}_v h )} \ = \ \frac{|\det \, S'|_v^2 \ \alpha_v(S', S') }{|\det \, S|_v^2 \ \alpha_v(S, S)}.
\end{equation}

If $v \neq p$, then $\lie{L}_v \simeq L_v$. In particular, we may choose the bases $\mathbf x$ and $\mathbf x'$ so that $S = S'$ in above display, and from this \textit{(i)} follows immediately. 

To prove \textit{(ii)}, we apply a standard calculation relating orbital integrals and Whittaker functionals, cf.\ the proof of \cite[Lemma 10.4]{KRunnglob}, and recall that here $p$ is inert:
\begin{align*}
 \int_{H_T(\Q_p)} \ \varphi' \left( h_T^{-1} x \right) \mathrm d_p \, h_T  \ =& \ \gamma_p(\V^+)^2 \  L_p(1, \chi_k) \ \frac{|\mathtt a'( \mathbf e' \otimes f)|_p}{   | \det S'|^{2}_p \   |\mathtt b ( \mathbf c) |_p} \ W_{T,p}(e, 0, \Phi'_p)  &  \\
 =& \  \gamma_p(\V^+)^2 \ L_p(1, \chi_k) \ \frac{|\mathtt a( \mathbf e \otimes f)|_p}{   | \det S|^{2}_p \   |\mathtt b ( \mathbf c) |_p} \ W_{T,p}(e, 0, \Phi'_p)  & [\text{by } \eqref{compAEqn} ] \\
 =& \ \gamma_p(\V^+)^2 \cdot \mathrm{vol}(K_{p}, \, \mathrm d_p h) \cdot \left( |\det S|_p^{2} \, \alpha_p(S, S) \right)^{-1}  \cdot W_{T,p}(e, 0, \Phi'_p) & [\text{by } \eqref{volKLEqn}]
\end{align*}
By definition, $L_p$ is the maximal lattice in the non-split hermitian space of dimension 2 over $k_p$, and so we may choose a basis so that $S = (\sm{p & \\ & 1} ) $. A direct computation using Hironaka's formula (\refThm{HiroThm}) gives the expression
\[ \alpha_p \left( ( \sm{p & \\ &1}) , \ (\sm{p & \\ &1}) \right) \ = \ p^{-1} (p+1)^2 , \]
which, when combined with equation preceding it, yields the statement of the lemma. 
\end{proof}
\end{lemma}
By combining the above calculation with the geometric computations of \refSec{GlobGeomSec}, we obtain our main theorem, relating the arithmetic degree of a special cycle $\Zed(T)$ to the Eisenstein series  
 \[\mathcal E(z, s) \ := \ \sum_{[L] \in \mathcal R_d} \ (C_{[L]})^{-1}  \ \mathcal E(z, s, [L]). \] 

\begin{corollary}  \label{mainCor}
Suppose  $T \in \Herm_2(o_k)$ is positive definite and $|\Diff(T) | \geq 1$. Then
 \[ \widehat\deg  \ \Zed(T) \ q^T = \ \frac{2 \, h(k)}{  |o_k^{\times}|}\  \mathcal E'_T(z,0). \]
 \begin{proof}
If $|\Diff(T)| \geq 2$, then the right hand side vanishes by \refThm{MainThmEis}(i), and the left hand side vanishes by \refLemma{ZSuppLemma}(iii). 

Next, suppose $\Diff(T) = \{ p \}$ for some $p|d$. Then comparing \refThm{MainThmGeom} and \refThm{MainThmEis}, 
\[ \widehat\deg \, \Zed(T)^{[L_0], [L]}  \ q^T \ = \ \frac{4 \, h(k)}{|o_k^{\times}| \,  2^{o(\Delta) } } \ C_{[L]}^{-1} \ \mathcal E'(z, 0, [L])  \]
for any $[L_0] \in  \mathcal R_0$ and $[L] \in \mathcal R_d$. Note the right hand side is independent of $[L_0]$. Given a Hermitian space $V_0$ of signature $(1,0)$ such that $\inv_p (V_0) = 1$ for all inert $p$, there is a single genus of self dual lattices in $V_0$; conversely if $L_0$ is self-dual, then $\inv_p(L_0 \otimes \Q) = 1$ for all inert $p$. Thus, by counting the possibilities of the invariants of $V_0$, 
\[ \# \mathcal R_0 \ = \ \# \{ V_0 \ | \ \inv_p V_0 = 1 \text{ for all } p \text{ inert} \}_{ \big / \textrm{isom.}}\ = \ 2^{o(d_{k}) - 1},\]
and so
\begin{align*}
 \widehat\deg \, \Zed(T) \ q^T \  =&\  \sum_{[L] \in \mathcal R_d } \  \sum_{[L_0] \in \mathcal R_0}  \ \widehat\deg \, \Zed(T)^{[L_0], [L]} \ q^T  \\
 =& \  2^{o(\Delta) - 1} \  \sum_{[L] \in \mathcal R_d} \widehat\deg \, \Zed(T)^{[L_0], [L]} \ q^T  \\
 =& \ \frac{2 \, h(k)}{|o_k^{\times}|} \ \mathcal E'(z, 0) .
 \end{align*}
Finally, when $\Diff(T) = \{ p \}$ for some $p \nmid d$, the desired result is exactly  \cite[Theorem 11.9]{KRunnglob} applied to the case at hand. 
  \end{proof}
\end{corollary}

 \end{document}